\documentclass[12pt,reqno]{amsart}
\usepackage[a4paper]{geometry}
\usepackage{amsthm,hyperref,lmodern,mathrsfs,amssymb,amsmath,amsfonts}
\usepackage{enumerate}
\usepackage[latin1]{inputenc}
\usepackage{versions}

\usepackage{color}

\allowdisplaybreaks

\numberwithin{equation}{section}

\newenvironment{Proof}{\removelastskip\par\medskip
\noindent{\em Proof.}
\rm}{\nolinebreak\hfill\rule{2mm}{2mm}\medbreak\par}

\newenvironment{Proof_th}{\removelastskip\par\medskip
\noindent{\em Alternative proof of Theorem \ref{theo:BL}.}
\rm}{\nolinebreak\hfill\rule{2mm}{2mm}\medbreak\par}

\newenvironment{Pf_th}{\removelastskip\par\medskip
\noindent{\em Proof of Theorem \ref{thm:asym-BL}.}
\rm}{\nolinebreak\hfill\rule{2mm}{2mm}\medbreak\par}

\newtheorem{theo}{Theorem}[section]
\newtheorem{lemme}[theo]{Lemma}
\newtheorem{prop}[theo]{Proposition}

\def\cE{{\mathcal{E}\ \!\!}}
\def\E{{\mathbb{E}\ \!\!}}

\def\N{{\mathbb{N}\ \!\!}}
\def\R{{\mathbb{R}}}
\def\L{{\mathcal{L}\ \!\!}} 
\def\LL{{L \ \!\!}} 

\def\D{{\mathcal{D}\ \!\!}}

\def\P{{\mathbb{P}\ \!\!}}
\def\Q{{Q\ \!\!}} 

\def\C{{\mathcal{C}\ \!\!}}
\def\M{{\mathcal{M}\ \!\!}}

\def\FF{{\mathcal{F}\ \!\!}}

\def\Var{{\mathrm{{\rm Var}}}}

\def\Cov{{\mathrm{{\rm Cov}}}}

\def\and{{\mathrm{{\rm and}}}}

\def\ve{{\varepsilon\ \!\!}}
\def\eps{{\varepsilon\ \!\!}}

\begin{document}

\title[Intertwinings and generalized Brascamp-Lieb inequalities]{Intertwinings and generalized Brascamp-Lieb inequalities}

\author{Marc~Arnaudon} %
\address[M.~Arnaudon]{UMR CNRS 5251, Institut de Math\'ematiques de Bordeaux, Universit\'e Bordeaux 1, France} %
\email{\url{mailto:marc.arnaudon(at)math.u-bordeaux1.fr}} %
\urladdr{\url{http://www.math.u-bordeaux1.fr/~marnaudo/}}

\author{Michel~Bonnefont} %
\address[M.~Bonnefont]{UMR CNRS 5251, Institut de Math\'ematiques de Bordeaux, Universit\'e Bordeaux 1, France} %
\thanks{MB is partially supported by the French ANR-12-BS01-0013-02 HAB project}
\email{\url{mailto:michel.bonnefont(at)math.u-bordeaux1.fr}} %
\urladdr{\url{http://www.math.u-bordeaux1.fr/~mibonnef/}}

\author{Ald\'eric~Joulin} %
\address[A.~Joulin]{UMR CNRS 5219, Institut de Math\'ematiques de Toulouse, Universit\'e de Toulouse, France}%
\thanks{AJ is partially supported by the French ANR-12-BS01-0019 STAB project}
\email{\url{mailto:ajoulin(at)insa-toulouse.fr}}%
\urladdr{\url{http://perso.math.univ-toulouse.fr/joulin/}}



\maketitle

\begin{abstract}
We continue our investigation of the intertwining relations for Markov semigroups and extend the results of \cite{bj} to multi-dimensional diffusions. In particular these formulae entail new functional inequalities of Brascamp-Lieb type for log-concave distributions and beyond. Our results are illustrated by some classical and less classical examples.
\end{abstract}

\section{Introduction}
\label{sect:intro}
The goal of this paper is to further complement our understanding of the intertwining relations between gradients and Markov semigroups and to explore its consequences in terms of functional inequalities of spectral flavour. These identities might be seen as a generalization/extension of the (sub-)commutation relations emphasized in the early eighties by Bakry and \'Emery in \cite{bakry_emery} via the well-known $\Gamma_2$ theory. After a preliminary study on this topic in the discrete case of integer-valued birth and death processes \cite{cj} and followed by its continuous counterpart through one-dimensional diffusions \cite{bj}, we investigate in the present notes the case of reversible and ergodic multi-dimensional diffusions on $\R^d$, $d\geq 2$, with generator of the type
$$
\LL f := \Delta f - (\nabla V) ^T \, \nabla f ,
$$
where $V$ is a smooth potential on $\R^d$ satisfying some nice conditions. Such processes admit as their unique invariant probability measure $ \mu$ the probability measure with Lebesgue density proportional to $e^{-V}$ on $\R^d$. \vspace{0.1cm}

The principle of the intertwining technique is the following: differentiate a given smooth Markov semigroup and write it, when it is possible, as an alternative semigroup acting on this derivative. In contrast to the one-dimensional case where the derivative of a function is still a function, the (Euclidean) gradient of a real-valued function defined on $\R^d$ is not a function and the resulting semigroup should act on gradients and not on functions. This apparently insignificant remark makes rather delicate the analysis of the desired intertwining identities in the multi-dimensional case since the two semigroups are, in essence, rather different. \vspace{0.1cm}

Dealing with the Bakry-\'Emery theory, the main protagonist appearing in the usual sub-commutation relations is the so-called \textit{carr\'e du champ (de gradient)} operator, thus fixing the underlying gradient of interest. The assumption required to obtain from these sub-commutation relations functional inequalities such as Poincar\'e or log-Sobolev inequalities is the strong convexity of the potential $V$. The interested reader is referred to the recent monograph \cite{BGL} for a nice introduction to this large body of work, with many references and credit. From the semi-classical analysis point of view, the intertwining relation occurs at the level of the generators. The resulting operator corresponds to the so-called Witten Laplacian acting on 1-forms which has been studied extensively by Helffer and Sj\"ostrand to establish for instance the decay of correlation in models arising in statistical mechanics, cf. \cite{helffer_book}. In the present paper, our main idea is to consider weighted, or ``distorted", gradients, i.e. gradients parametrized by invertible square matrices. This large degree of freedom leads us to obtain, as a consequence of these intertwinings, a family of new functional inequalities extending the classical Brascamp-Lieb inequality satisfied under the strict convexity assumption of the potential $V$. Among other things, an important point in our results resides in the fact that the potential $V$ is no longer required to be convex, allowing us to consider in the analysis probability measures $\mu$ which are non-necessarily log-concave. \vspace{0.1cm}

Let us describe the content of the paper. In Section \ref{sect:intert} we recall basic material on diffusion operators and state our main result, Theorem \ref{theo:intert}, in which an intertwining relation is obtained between a weighted gradient and two different semigroups: the first one is the underlying Markov semigroup acting on functions whereas the second one is a semigroup of Feynman-Kac type acting on weighted gradients. Sections \ref{sect:BL} and \ref{sect:asym_BL} are devoted to analyze the consequences of the intertwining approach in terms of functional inequalities. More precisely we obtain Brascamp-Lieb type inequalities involving the variance or the covariance in the left-hand-side. We call the first ones involving the variance generalized Brascamp-Lieb inequalities since the energy term in the right-hand-side is modified, extending \textit{de facto} the classical one usually obtained without introducing a weight in the gradient. The second inequalities of interest are called asymmetric Brascamp-Lieb inequalities in \cite{menz_otto} and involve the covariance of two functions decorrelated in the sense that the right-hand-side of the inequality is the product of two conjugate $L^p$ norms of their weighted gradients. We also derive new lower bounds on the spectral gap of the diffusion operator and among them, one is an extension to the multi-dimensional setting of the famous variational formula of Chen and Wang derived in the one-dimensional case \cite{chen_wang}. Finally, to convince the reader of the relevance of the intertwining approach, we illustrate our results by revisiting some classical and less classical examples.

\section{Intertwinings}\label{sect:intert}

Let $\C ^\infty (\R ^d , \R)$ be the space of infinitely differentiable real-valued functions on the Euclidean space $(\R^d , \vert \cdot \vert )$, $d\geq 2$, and let $\C _0 ^\infty (\R ^d, \R)$ be the subspace of $\C ^\infty (\R ^d , \R)$ of compactly supported functions. Denote $\Vert \cdot \Vert _\infty$ the essential supremum norm with respect to the Lebesgue measure. In this paper we consider the second-order diffusion operator defined on $\C ^\infty (\R ^d , \R)$ by
$$
\LL f := \Delta f - (\nabla V) ^T \, \nabla f ,
$$
where $V$ is a smooth potential on $\R^d$ whose Hessian matrix $\nabla \nabla V$ is, with respect to the space variable, uniformly bounded from below (in the sense of symmetric matrices). Above $\Delta$ and $\nabla$ stand respectively for the Euclidean Laplacian and gradient and the symbol $^T$ means the transpose of a column vector (or a matrix). Let $\Gamma$ be the \textit{carr\'e du champ} operator which is the bilinear symmetric form defined on $\C ^\infty (\R ^d , \R) \times \C ^\infty (\R ^d , \R)$ by
$$
\Gamma (f,g) := \frac{1}{2} \, \left( \LL (fg) - f \, \LL g - g \, \LL f \right) = (\nabla f) ^T \, \nabla g .
$$
If $e^{-V}$ is Lebesgue-integrable on $\R^d$, a condition which will be assumed throughout the whole paper, then we denote $ \mu$ the probability measure with Lebesgue density proportional to $e^{-V}$ on $\R^d$. The operator $\LL$, which satisfies on $\C _0 ^\infty (\R ^d , \R)$ the invariance property
$$
\int_{\R^d} \LL f \, d\mu = 0,
$$
and also the symmetry with respect to $\mu$, that is, for every $f,g \in \C _0 ^\infty (\R ^d , \R)$,
$$
\cE _\mu (f,g)  := \int_{\R ^d} \Gamma (f,g) \, d\mu = - \int_{\R ^d } f\, \LL g \, d\mu = - \int_{\R ^d } \LL f \, g \, d\mu = \int_{\R ^d } (\nabla f) ^T \, \nabla g \, d\mu ,
$$
is non-positive on $\C _0 ^\infty (\R ^d , \R)$. Since the Euclidean state space is complete, the operator is essentially self-adjoint, i.e. it admits a unique self-adjoint extension (still denoted $\LL$) with domain $\D (\LL) \subset L^2 (\mu)$ in which the space $\C _0 ^\infty (\R ^d , \R)$ is dense for the norm induced by $\LL$, that is,
$$
\Vert f\Vert _{\D (L)} := \sqrt{\Vert f \Vert ^2 _{L^2 (\mu)} + \Vert L f \Vert ^2 _{L^2 (\mu)}}.
$$
By the spectral theorem it generates a unique strongly continuous symmetric semigroup $(P_t)_{t\geq 0}$ on $L^2 (\mu)$ such that for every function $f\in L^2 (\mu)$ and every $t>0$ we have $P_t f \in \D (L)$ and
$$
\partial _t P_t f = P_t Lf = L P_t f .
$$
Here $\partial_t$ denotes the partial derivative with respect to the time parameter $t$. Moreover the ergodicity property holds: for every $f\in L^2 (\mu)$, we have the following convergence in $L^2 (\mu)$:
$$
\lim_{t \to \infty} \, \left \Vert P_t f - \mu (f) \right \Vert _{L^2 (\mu)} = 0,
$$
where $\mu (f)$ stands for the integral of $f$ with respect to $\mu$. Finally, the closure $(\cE _\mu , \D (\cE _\mu))$ of the bilinear form $(\cE _\mu,\C _0 ^\infty (\R ^d , \R))$ is a Dirichlet form on $L^2(\mu)$ and by the spectral theorem we have the dense inclusion $\D (\LL) \subset \D (\cE _\mu)$ for the norm induced by $\cE_\mu$. \vspace{0.1cm}

Dealing with stochastic processes, the operator $\LL$ is the generator of the process corresponding to the unique strong solution $(X_t ^x)_{t\geq 0}$ of the following Stochastic Differential Equation (in short SDE),
$$
dX_t ^x = \sqrt{2} \, dB_t - \nabla V (X_t ^x) \, dt, \quad X_0 ^x = x \in \R ^d,
$$
where $(B_t)_{t\geq 0}$ is a standard $\R^d$-valued Brownian motion on a given filtered probability space $(\Omega , \FF , (\FF_t)_{t\geq 0}, \P)$. By \cite{bakry:non-exp} the lower bound assumption on $\nabla\nabla V$ entails that the process is non-explosive or, in other words, it has an infinite lifetime, and the finiteness of the invariant probability measure $\mu$ reflects its positive recurrence. In terms of semigroups we have $P_t f(x) = \E [f(X_t ^x)]$ for every $f\in \C _0 ^\infty (\R^d, \R)$. When the potential $V(x) = \vert x\vert ^2 /2$ then the process $(X_t ^x)_{t\geq 0}$ is nothing but the Ornstein-Uhlenbeck process starting from $x$ and whose invariant measure is the standard Gaussian distribution, say $\gamma$. In particular we have the well-known intertwining between gradient and semigroup
$$
\nabla P_t f = e^{-t} \, \Q_t (\nabla f),
$$
where in the right-hand-side the semigroup $(Q_t)_{t\geq 0}$ acts on smooth vector fields coordinate by coordinate. The exponential term can be interpreted as a curvature term and is associated to the ergodicity of the process. Actually such a commutation relation is reminiscent of the intertwining at the level of the generator, that is, if $I$ stands for the identity operator, then for every $\C ^\infty (\R^d , \R)$,
$$
\nabla \LL f = (\L - I) \, (\nabla f) .
$$
Once again the expression $\L (\nabla f)$ has to be understood coordinate-wise or, in other words, $\L$ is considered (and will be considered along the paper) as a diagonal matrix operator,
\[
\L (\nabla f) = \left(
\begin{array}{ccc}
\LL & & \\
& \ddots & \\
& & \LL
\end{array}
\right) \, (\nabla f) .
\]
Coming back to the case of a general smooth potential $V$, we wonder if we can obtain a somewhat similar intertwining relation. Doing the computations leads us to the following identity:
$$
\nabla \LL f = (\L - \nabla \nabla V) \, (\nabla f) .
$$
Now our idea is to bring a distortion of the gradient by introducing a weight given by a smooth invertible matrix $A : \R^d \to \M _{d\times d}(\R)$, that is
\begin{equation}
\label{eq:intert_gen}
A \, \nabla \LL f = \L _A ^{M_A} (A\nabla f) ,
\end{equation}
where $\L_A ^{M_A}$ is the matrix Schr\"odinger operator acting on $\C ^\infty (\R^d , \R^d)$, the space of smooth vector fields $F: \R^d \to \R^d$, as
\begin{eqnarray*}
\L _A ^{M_A} F & = & A\, (\L - \nabla\nabla V ) \, (A^{-1} \, F) \\
& = & \L F + 2 \, A \, \nabla A^{-1} \, \nabla F - ( A \, \nabla \nabla V \, A^{-1} - A \, \L \, A^{-1} ) \, F \\
& = & (\L_A -M_A ) \, F ,
\end{eqnarray*}
where $\L_A$ denotes the (possible) non-diagonal matrix operator acting on $\C ^\infty (\R^d , \R^d)$ as
$$
\L _A F :=  \L F + 2 \, A \, \nabla A^{-1} \, \nabla F ,
$$
and $M_A$ is the matrix corresponding to the multiplicative (or zero-order) operator
$$
M_A := A \, \nabla \nabla V \, A^{-1} - A \, \L A^{-1} .
$$
Above the gradient $\nabla F$ of the vector field $F$ is the column vector whose coordinates are the $\nabla F_i$, $i=1, \ldots, d$ and if $A^{-1} = (a^{i,j})_{i,j = 1, \ldots, d}$ then $\nabla A^{-1}$ is the matrix of gradients $(\nabla a^{i,j})_{i,j = 1,\ldots, d}$. \vspace{0.1cm}

Let us turn our attention to the question of symmetry. Denote the positive-definite matrix $S := (A \, A^T)^{-1}$ and let $L^2(S,\mu)$ be the weighted $L^2$ space consisting of vectors fields $F: \R^d \to \R^d$ such that
$$
\int_{\R^d} F^T \, S \, F \, d\mu  < \infty .
$$
Given $f,g \in \C_0 ^\infty( \R^d , \R)$, we have
\begin{eqnarray*}
\int_{\R^d} (\L _A ^{M_A} (A \, \nabla f )) ^T \, S \, A \, \nabla g \, d\mu & = & \int_{\R^d} (\nabla L f) ^T \, \nabla g \, d\mu \\
& = & - \int_{\R^d} \LL f \, \LL g \, d\mu ,
\end{eqnarray*}
so that $- \L _A ^{M_A}$ is a symmetric and non-negative operator on the subspace of weighted gradients $\nabla _A := \{ A \, \nabla f : f\in \C_0 ^\infty( \R^d , \R) \} \subset L^2 (S, \mu)$. However to ensure these two properties on the bigger space $\C _0 ^\infty (\R ^d , \R ^d)$ of compactly supported smooth vectors fields, one needs an additional assumption on the matrix $A$. On the one hand we have for every $F,G \in \C _0 ^\infty (\R ^d , \R ^d)$,
\begin{eqnarray*}
- \int_{\R^d} (\L _A F )^T \, S \, G \, d\mu & = & \int_{\R^d} (\nabla F) ^T \, S \, \nabla G \, d\mu \\
& & + \int_{\R^d} (\nabla F) ^T \, \left( \nabla S - 2 \, (\nabla A^{-1})^{T} \, A^T \, S \right) \, G \, d\mu ,
\end{eqnarray*}
so that the operator $- \L _A$ satisfies the desired properties of symmetry and non-negativity as soon as the following matrix equation holds:
\begin{equation}
\label{eq:S}
\nabla S = 2 \, (\nabla A^{-1})^{T} \, A^T \, S ,
\end{equation}
thus requiring the symmetry of the matrix $(A^{-1})^T \, \nabla A^{-1}$. On the other hand the multiplicative operator $M_A$, seen as an operator on $L^2 (S,\mu)$, is symmetric on $\C _0 ^\infty (\R ^d , \R ^d)$ as soon as the matrix $S \, M_A$ is symmetric, which is equivalent to the symmetry of the matrix $(A^{-1})^T \, \L A^{-1}$. Note that its is also equivalent to the symmetry of the matrix $A^{-1} \, M_A \, A $ which rewrites as
$$
A^{-1} \, M_A \, A = \nabla \nabla V - \L A^{-1} \, A.
$$
Actually we can show that the equation \eqref{eq:S} is equivalent to the three conditions above since by integration by parts we have the identities
\begin{eqnarray*}
\lefteqn{\int_{\R^d} (\nabla F) ^T \, \left( \nabla S - 2 \, (\nabla A^{-1})^{T} \, A^T \, S \right) \, G \, d\mu } \\
& = & \int_{\R^d} (\nabla F) ^T \, \left(  ( A^{-1})^{T} \nabla  A^{-1}  -  \, (\nabla A^{-1})^{T} \, A^{-1}\right) \, G \, d\mu  \\
& = & \int_{\R^d} F ^T \, \left( (\L A^{-1}) ^T \, A^{-1} - (A^{-1})^T \, \L A^{-1} \right) \, G \, d\mu \\
& & - \int_{\R^d} F ^T \, \left( (A^{-1})^T \, \nabla A^{-1} - (\nabla A^{-1})^T \, A^{-1} \right) \, \nabla G \, d\mu .
\end{eqnarray*}
Certainly, even under the symmetry assumption on $A^{-1} \, M_A \, A$, the matrix $M_A$ itself has no reason to be symmetric, unless $S$ is a multiple of the identity, a particular case which will be exploited in Sections \ref{sect:asym_BL} and \ref{sect:ex}. Nevertheless since the matrices $A^{-1} \, M_A \, A$ and $M_A$ are similar they have the same eigenvalues and therefore the matrix $M_A$ is diagonalizable. \vspace{0.1cm}

Let us summarize our situation. We consider two operators acting on $\C _0 ^\infty (\R ^d , \R ^d) \subset L^2 (S,\mu)$: the Schr\"odinger operator $\L _A ^{M_A}$ and the operator $\L _A$. Both are symmetric on $\C _0 ^\infty (\R ^d , \R ^d)$ if and only if one of the following equivalent assertions is satisfied:
\begin{itemize}
\item[$(i)$] the matrix $(A^{-1})^T \, \nabla A^{-1}$ is symmetric.
\item[$(ii)$] the matrix $(A^{-1})^T \, \L \, A^{-1}$ is symmetric.
\item[$(iii)$] the matrix $S \, M_A$ is symmetric.
\item[$(iv)$] the matrix $A^{-1} \, M_A \, A $ is symmetric.
\end{itemize}
Moreover under these assumptions $-\L _A$ is always non-negative on $\C _0 ^\infty (\R ^d , \R ^d)$ whereas $-\L _A ^{M_A}$ is non-negative as soon as the matrix $S \, M_A$ is positive semi-definite. Finally, when restricted to the space of weighted gradients $\nabla_A$, then $-\L _A ^{M_A}$ is symmetric and non-negative without any additional assumption on the invertible matrix $A$. \vspace{0.1cm}

Assume now that the matrix $S \, M_A$ is symmetric and bounded from below, in the sense of symmetric matrices, by $\lambda \, S$, where $\lambda $ is some real parameter independent from the space variable. In other words the symmetric matrix $A^{-1} \, M_A \, A $ is uniformly bounded from below by $\lambda \, I$ (this observation will be used many times all along the paper). Following \cite{Strichartz}, one can show that the Schr\"odinger operator $\L _A ^{M_A}$ is essentially self-adjoint on $\C^\infty _0 (\R ^d, \R^d)$ and thus admits a unique self-adjoint extension (still denoted $\L_A ^{M_A}$) with domain $\D (\L_A ^{M_A}) \subset L^2 (S,\mu)$. The associated semigroup is denoted $(\Q_{t,A} ^{M_A})_{t\geq 0}$ and is thus the unique semigroup solution to the $L^2$ Cauchy problem, that is,
\[
\left \{
\begin{array}{cccc}
\partial _t F & = & \L _A ^{M_A} F & \\
F(\cdot , 0) & = & G, & G\in L^2 (S,\mu),
\end{array}
\right.
\]
where $F(\cdot, t)$ is required to be in $L^2(S,\mu)$ for every $t>0$. Actually, it can be shown as follows that this uniqueness result in $L^2( S,\mu)$ holds in full generality. In the sequel we say that a smooth vector field $F: \R^d \times [0,\infty) \to \R^d$ is locally bounded in $L^2 (S,\mu)$ if
$\sup_{t\in [0,T]} \Vert F(\cdot , t) \Vert _{L^2 (S,\mu)} < \infty$ for every finite time horizon $T>0$.
\begin{prop}
\label{prop:cauchy_pb}
Assume that the matrix $A^{-1} \, M_A \, A $ is symmetric and uniformly bounded from below. Let $F$ be a smooth vector field which is locally bounded in $L^2(S,\mu)$. If such a $F$ is solution to the $L^2$ Cauchy problem
\[
\left \{
\begin{array}{cccc}
\partial _t F & = & \L _A ^{M_A} F \\
F(\cdot , 0) & = & G, &  G\in L^2 (S,\mu),
\end{array}
\right.
\]
then we have $F(\cdot , t ) = \Q_{t,A} ^{M_A} G$ for every $t \geq 0$.
\end{prop}
\begin{proof}
We adapt the argument emphasized by Li in \cite{Li} in the context of the heat equation on complete Riemannian manifolds.
By linearity it is sufficient to show that 0 is the unique $F$, smooth and locally bounded in $L^2 (S,\mu)$, solution to the $L^2$ Cauchy problem
\[
\left \{
\begin{array}{ccc}
\partial _t F & = & \L _A ^{M_A} F \\
F(\cdot , 0) & = & 0 .
\end{array}
\right.
\]
Denoting $F$ such a solution and replacing $F(\cdot, t)$ by $e^{- \inf \rho_A \, t} \, F(\cdot, t)$, where $\rho_A : \R ^d \to \R$ stands for the smallest eigenvalue of the matrix $A^{-1} \, M_A \, A $ (recall that $\inf \rho_A > -\infty$ according to our assumption), let us assume without loss of generality that $A^{-1} \, M_A \, A $ is positive semi-definite, i.e. $S\, M_A$ is. Letting $\phi \in \C _0 ^\infty (\R^d , \R)$ we have for every $\tau >0$,
\begin{eqnarray*}
\int_0^\tau \int _{\R^d} \phi^2 \, F^T \, S \, \L _A F \, d\mu \, dt & = & \int_0^\tau \int _{\R^d} \phi^2 \, \left( F^T \, S \, \L _A ^{M_A} F + F^T \, S \, M_A \, F \right) \, d\mu \, dt \\
& = & \int_0^\tau  \int _{\R^d} \phi^2 \, \left(\frac{1}{2} \, \partial_t ( F^T \, S \, F ) + F^T \, S \, M_A \, F \right) \, d\mu \, dt \\
&\geq & \frac{1}{2} \, \int_{\R^d} \phi^2 \, F(\cdot , \tau ) ^T \, S \, F(\cdot , \tau ) \, d\mu.
 \end{eqnarray*}
Now by integration by parts and symmetry of $\L _A$, we have
\begin{eqnarray*}
\int_0^\tau  \int _{\R ^d} \phi^2 \, F ^T \, S \, \L _A F \, d\mu \, dt & = & - \int_0^\tau  \int _{\R^d} (\nabla (\phi^2 \, F)) ^T \, S \, \nabla F \, d\mu \, dt \\
& = & - \int_0^\tau  \int _{\R^d} 2\, \phi \, (\nabla \phi \, F)^T \, S \, \nabla F \, d\mu \, dt \\
& & - \int_0^\tau  \int _{\R^d} \phi^2 \, (\nabla  F) ^T \, S \, \nabla F \, d\mu \, dt .
\end{eqnarray*}
Moreover by Cauchy-Schwarz' inequality, we have for every $\lambda>0$,
\begin{eqnarray*}
2 \, \left \vert \phi \, (\nabla \phi \, F) ^T \, S \, \nabla F \right \vert & = & 2 \, \left \vert (\nabla \phi \, F) ^T \, S \, \phi \, \nabla F \right \vert \\
& \leq & \lambda \, (\nabla \phi \, F) ^T \, S \, \nabla \phi \, F + \frac{1}{\lambda} \, \phi \, (\nabla F) ^T \, S \, \phi \, \nabla F \\
& = & \lambda \, |\nabla \phi | ^2 \, F ^T \, S \, F + \frac{1}{\lambda} \, \phi^2 \, (\nabla F)^T \, S \, \nabla F .
\end{eqnarray*}
Therefore we obtain from the above inequalities,
\begin{eqnarray*}
\frac{1}{2} \, \int_{\R^d} \phi^2 \, F(\cdot , \tau ) ^T \, S \, F(\cdot , \tau) \, d\mu & \leq & \int_0^\tau  \int _{\R ^d} \phi^2 \, F ^T \, S \, \L _A F \, d\mu \, dt \\ & \leq & \left( \frac{1}{\lambda} - 1 \right) \, \int_0^\tau  \int _{\R^d} \phi^2 \, (\nabla F)^T \, S \, \nabla F \, d\mu \, dt \\
& & + \lambda \, \int_0^\tau  \int _{\R^d} |\nabla \phi | ^2 \, F ^T \, S \, F \, d\mu \, dt .
\end{eqnarray*}
In particular for $\lambda=2$ we get
\begin{eqnarray*}
\frac{1}{2} \, \int_{\R^d} \phi^2 \, F(\cdot , \tau ) ^T \, S \, F(\cdot , \tau) \, d\mu & \leq & - \frac{1}{2} \, \int_0^\tau  \int _{\R^d} \phi^2 \, (\nabla F)^T \, S \, \nabla F \, d\mu \, dt \\
& & + 2 \, \int_0^\tau  \int _{\R^d} |\nabla \phi | ^2 \, F ^T \, S \, F \, d\mu \, dt .
\end{eqnarray*}
Finally by completeness there exists a sequence of $[0,1]$-valued functions $(\phi_n)_{n\in \N} \subset \C ^\infty _0 (\R ^d , \R)$ such that $\phi_n \uparrow 1$ pointwise and $\Vert \vert \nabla \phi_n \vert \Vert_\infty \to 0$ as $n\to \infty$. Plugging this sequence of functions in the latter inequality and letting $n\to \infty$ gives both
$$
\int_{\R^d} F(\tau,\cdot ) ^T \, S \, F(\tau,\cdot) \, d\mu = 0 \quad \mbox{ and } \quad \int_0^\tau  \int _{\R^d} (\nabla F)^T \, S \, \nabla F \, d\mu \, dt = 0 ,
$$
hence $F =0$. The proof is achieved.
\end{proof}

Now we are in position to state an intertwining relation between gradient and semigroup.
\begin{theo}
\label{theo:intert}
Assume that the matrix $A^{-1} \, M_A \, A $ is symmetric and uniformly bounded from below. Then the following intertwining relation is satisfied: for every $f\in C_0 ^\infty (\R^d , \R)$,
$$
A \, \nabla P_t f = \Q_{t,A} ^{M_A} (A\, \nabla f ), \quad t \geq 0.
$$
\end{theo}
\begin{proof}
Although the proof is somewhat similar to that provided in \cite{bj} in the one-dimensional case, we recall the argument for the sake of completeness. The idea is to show that the vector field defined by $F(\cdot , t) := A \, \nabla P_t f$ is the unique smooth and locally bounded in $L^2(S,\mu)$ solution to the $L^2$ Cauchy problem
\[
\left \{
\begin{array}{cccc}
\partial _t F & = & \L _A ^{M_A} F \\
F(\cdot , 0) & = & G ,
\end{array}
\right.
\]
with $G := A \, \nabla f \in L^2 (S,\mu)$. First since $A$ is assumed to be smooth and by the ellipticity property of the operator $L$, the vector field $F$ is smooth on $\R^d \times (0,\infty)$. Moreover we have
\begin{eqnarray*}
\int_{\R ^d} F(\cdot ,0)^T \, S \, F(\cdot ,0) \, d\mu & = & \int_{\R ^d} (A \, \nabla f )^T \, S \, A \, \nabla f \, d\mu \\
& = & \int_{\R ^d} \vert \nabla f \vert ^2 \, d\mu ,
\end{eqnarray*}
which is finite since $f\in C_0 ^\infty (\R^d , \R)$, hence $F(\cdot ,0) \in L^2 (S,\mu)$. Now for every $t>0$ we have
\begin{eqnarray*}
\int_{\R ^d} F(\cdot ,t)^T \, S \, F(\cdot ,t) \, d\mu & = & \int_{\R ^d} (A \, \nabla P_t f )^T \, S \, A \, \nabla P_t f \, d\mu \\
& = & \int_{\R ^d} \vert \nabla P_t f \vert ^2 \, d\mu \\
& = & - \int_{\R ^d} P_t f \, \LL P_t f \, d\mu .
\end{eqnarray*}
Differentiating with respect to the time parameter and using integration by parts yield
$$
\partial _t \int_{\R ^d} F(\cdot ,t)^T \, S \, F(\cdot ,t) \, d\mu = -2 \, \int_{\R ^d} ( \LL P_t f)^2 \, d\mu \leq 0,
$$
hence the functional above is non-increasing in time and one deduces that
$$
\int_{\R ^d} F(\cdot ,t)^T \, S \, F(\cdot ,t) \, d\mu \leq \int_{\R ^d} F(\cdot ,0)^T \, S \, F(\cdot ,0) \, d\mu < \infty ,
$$
so that $F$ is locally bounded in $L^2 (S, \mu)$. Finally by the intertwining relation \eqref{eq:intert_gen} at the level of the operators, we obtain
$$
\partial _t F = A \, \nabla \LL P_t f = \L _A ^{M_A} (A \, \nabla P_t f ) = \L _A ^{M_A} F ,
$$
and the desired result follows by Proposition \ref{prop:cauchy_pb}.
\end{proof}

\section{Generalized Brascamp-Lieb inequalities and spectral gap}
\label{sect:BL}
Let us turn to the potential consequences of the intertwining relation of Theorem \ref{theo:intert} in terms of functional inequalities. In particular we focus our attention on an inequality due to Brascamp and Lieb \cite{brascamp_lieb}, known nowadays as the Brascamp-Lieb inequality. Under the notation of the preceding part, it is stated as follows: if the symmetric matrix $\nabla \nabla V $ is positive definite then for every sufficiently smooth function $f$ we have
\begin{equation}
\label{eq:BL}
\Var_\mu(f) \leq \int_{\R^d} (\nabla f) \, ^T \, (\nabla \nabla V) ^{-1} \, \nabla f \, d\mu ,
\end{equation}
where $\Var_\mu(f)$ stands for the variance of the function $f$ under $\mu$, that is
$$
\Var_\mu (f) := \mu (f^2) - \mu (f) ^2 .
$$
The extremal functions of the latter inequality are given by $f = \alpha ^T \, \nabla V$ with $\alpha$ some constant vector in $\R^d$. Among the interesting features of such an inequality, one of them is its connection with spectral theory. Indeed under the strong convexity condition $\nabla \nabla V \geq \lambda \, I $ where $\lambda >0$, i.e. the Euclidean version of the so-called Bakry-\'Emery criterion is satisfied (we will turn to this criterion in a moment), the Brascamp-Lieb inequality implies the Poincar\'e inequality with constant $\lambda$, i.e.,
\begin{equation}
\label{eq:Poincare}
\lambda \, \Var_\mu(f) \leq \int_{\R^d} \vert \nabla f \vert ^2 \, d\mu ,
\end{equation}
an inequality giving an exponential rate of convergence to equilibrium in $L^2 (\mu)$ of the underlying semigroup $(P_t)_{t\geq 0}$, that is for every $f\in L^2 (\mu)$,
$$
\left \Vert P_t f - \mu (f) \right \Vert _{L^2 (\mu)} \leq e^{-\lambda t} \, \left \Vert f - \mu (f) \right \Vert _{L^2 (\mu)} , \quad t \geq 0.
$$
The optimal constant $\lambda$ in \eqref{eq:Poincare} is nothing but the spectral gap in $L^2 (\mu)$, say $\lambda_1 (-\LL , \mu)$, of the operator $-\LL$. In practise there exists a spectral gap as soon as the potential $V$ is convex, cf. \cite{kls, bobkov} (the measure $\mu$ is said to be log-concave) or, provided a perturbation argument is used in the Bakry-\'Emery criterion, strictly convex at infinity, both involving rather bad constants with respect to the dimension. \vspace{0.1cm}

Due to the development of modern techniques such as optimal transportation and functional inequalities of geometrical inspiration, Brascamp-Lieb type inequalities have attracted a lot of attention recently. On the one hand, inspired by the known fact that linearizing a transport cost inequality entails an inequality of Poincar\'e-type (the Brascamp-Lieb inequality can be seen as belonging to this class of inequalities), the authors in \cite{bgg,cordero} investigated new forms of transportation inequalities and consequently derived various positive lower bounds on the quantity
$$
\int_{\R^d} (\nabla f) \, ^T \, (\nabla \nabla V) ^{-1} \, \nabla f \, d\mu - \Var_\mu(f) .
$$
In other words they reinforced \eqref{eq:BL} by a remainder term. On the other hand some functional inequalities such as Pr\'ekopa-Leindler and Borrell-Brascamp-Lieb inequalities have been used in \cite{bob_ledoux3,bob_ledoux2} and also in \cite{bgg} to get dimensional refinements of \eqref{eq:BL}. Actually, although both approaches revealed to be convenient, the results emphasized are not really comparable - see the interesting discussion on this fact in \cite{bgg}. Let us also mention the articles \cite{harge, nguyen} who used $L^2$ methods of H\"ormander type \cite{hormander} to obtain these type of results, and also the recent \cite{milman_koles} which focuses on Riemannian manifolds with boundary. \vspace{0.1cm}

In the sequel we obtain a generalization of the Brascamp-Lieb inequality \eqref{eq:BL} by modifying the energy term in the right-hand-side. The starting point of our approach is somewhat similar to the classical method emphasized by Helffer and Sj\"ostrand with the so-called Witten Laplacian (the operator $\L - \nabla \nabla V$ acting on smooth vector fields in our context) and more precisely in \cite{helffer} in which Helffer obtains from convenient covariance identities the decay of correlation in models arising in statistical mechanics. See for instance \cite{helffer_book} for a nice overview of the topic. Below, our main contribution comes from the distortion of such representations by means of the intertwining relation of Theorem \ref{theo:intert}, allowing us a degree of freedom in the choice of the invertible matrix $A$, hence in the right-hand-side of \eqref{eq:BL}. Our result stands as follows.

\begin{theo}
\label{theo:BL}
Assume that the matrix $A ^{-1} \, M_A \, A$ is symmetric and positive definite. Then for every $f\in \C ^\infty _0 (\R^d , \R)$ we have the generalized Brascamp-Lieb inequality
\begin{equation}
\label{eq:BLA}
\Var_\mu(f) \leq \int_{\R^d} (\nabla f) \, ^T \, (A ^{-1} \, M_A \, A) ^{-1} \, \nabla f \, d\mu .
\end{equation}
\end{theo}
\begin{proof}
First let us assume that the smallest eigenvalue $\rho_A$ of the matrix $A ^{-1} \, M_A \, A$ is bounded from below by some positive constant. Since the Schr\"odinger operator $- \L _A ^{M_A}$ is essentially self-adjoint on $\C ^\infty _0 (\R^d, \R ^d)$ and bounded from below by $\inf \rho_A \, I$ in the sense of self-adjoint operators on $L^2(S,\mu)$, that is for every $F\in \C _0 ^\infty (\R^d, \R^d)$,
\begin{eqnarray*}
- \int_{\R ^d} F^T  \, S \, \L _A ^{M_A} F \, d\mu & = & \int_{\R ^d} (\nabla F) \, ^T \, S \, \nabla F \, d\mu + \int_{\R ^d} F^T \, S \, M_A \, F \, d\mu \\
& \geq & \inf \rho_A \, \int_{\R ^d} F^T \, S \, F \, d\mu,
\end{eqnarray*}
then it is invertible in $L^2 (S, \mu)$ and given $F \in \C ^\infty _0 (\R^d , \R ^d)$, the Poisson equation $-\L _A ^{M_A} G = F$ admits a unique solution $G \in L^2 (S,\mu)$ which can be written as
$$
G = \int_0 ^\infty \Q_{t,A} ^{M_A} F \, dt .
$$
Using ergodicity, we have by Theorem \ref{theo:intert} the variance representation
\begin{eqnarray}
\label{eq:var_id}
\nonumber \Var_\mu(f) & = & \int_{\R^d} f \, (f- \mu(f)) \, d\mu \\
\nonumber & = & - \int_{\R^d} \int_0 ^\infty f \, \LL P_t f \, dt \, d\mu \\
\nonumber & = & \int_0 ^\infty \int_{\R^d} (\nabla f) ^T \, \nabla P_t f \, d\mu \, dt \\
\nonumber & = & \int_0 ^\infty \int_{\R^d} (A \, \nabla f) \, ^T \, S \, A \, \nabla P_t f \, d\mu \, dt \\
\nonumber & = & \int_0 ^\infty \int_{\R^d} (A \, \nabla f) \, ^T \, S \, \Q_{t,A} ^{M_A} (A \, \nabla f) \, d\mu \, dt \\
            & = & \int_{\R^d} (A \, \nabla f) \, ^T \, S \, (-\L _A ^{M_A}) ^{-1} (A \, \nabla f) \, d\mu .
\end{eqnarray}
Since we have the reverse inequality
$$
(- \L _A ^{M_A}) ^{-1} = (-\L_A + M_A)^{-1} \leq M_A ^{-1} ,
$$
understood once again in the sense of self-adjoint operators in $L^2 (S, \mu)$, we obtain from \eqref{eq:var_id} the variance inequality
$$
\Var_\mu(f) \leq \int_{\R^d} (A \, \nabla f) \, ^T \, S \, M_A ^{-1} \, (A \, \nabla f) \, d\mu ,
$$
which rewrites as the desired generalized Brascamp-Lieb inequality. \vspace{0.1cm} \\
Now if the symmetric matrix $A ^{-1} \, M_A \, A$ is only positive definite, then an approximation procedure is required. To do so, since the Schr\"odinger operator $- \L_A ^{M_A}$ is only non-negative, we consider for every $\varepsilon >0$ the operator $\varepsilon \, I - \L _A ^{M_A}$ which has the desired property. In particular for every $F \in \C ^\infty _0 (\R^d , \R ^d)$, the Poisson equation $(\varepsilon \, I -\L _A ^{M_A}) \, G = F$ admits a unique solution $G = G_\varepsilon \in L^2 (S,\mu)$ given by
$$
G_\varepsilon = \int_0 ^\infty e^{-\varepsilon \, t} \, \Q_{t,A} ^{M_A} F \, dt .
$$
At the level of the non-negative operator $-\LL$ acting on functions, let $f\in \C ^\infty _0 (\R^d , \R)$ be such that $\mu (f) = 0$ and consider the unique centered solution $g_\varepsilon \in L^2 (\mu)$ to the Poisson equation $(\varepsilon \, I -\LL ) \, g_\varepsilon = f$ which also admits the integral representation
$$
g_\varepsilon = \int_0 ^\infty e^{-\varepsilon \, t} \, P_t f \, dt .
$$
Using the same method as before and with the help of Cauchy-Schwarz' inequality, the variance identity \eqref{eq:var_id} becomes
\begin{eqnarray*}
\Var_\mu(f) & = & \int_{\R^d} f \, (\varepsilon \, I - \LL ) \, g_\varepsilon \, d\mu \\
            & = & \varepsilon \, \int_{\R^d} f \, g_\varepsilon \, d\mu + \int_0 ^\infty e^{-\varepsilon t} \int_{\R^d} (A \, \nabla f) \, ^T \, S \, A \, \nabla P_t f \, d\mu \, dt \\
            & = & \varepsilon \, \int_{\R^d} f \, g_\varepsilon \, d\mu + \int_0 ^\infty e^{-\varepsilon t} \int_{\R^d} (A \, \nabla f) \, ^T \, S \, \Q_{t,A} ^{M_A} (A \, \nabla f) \, d\mu \, dt \\
            & \leq & \varepsilon \, \Vert f\Vert _{L^2 (\mu)} \, \Vert g_\varepsilon \Vert _{L^2 (\mu)} + \int_{\R^d} (A \, \nabla f) \, ^T \, S \, (\varepsilon \, I -\L _A ^{M_A}) ^{-1} (A \, \nabla f) \, d\mu \\
            & \leq & \varepsilon \, \Vert f\Vert _{L^2 (\mu)} \, \Vert g_\varepsilon \Vert _{L^2 (\mu)} + \int_{\R^d} (A \, \nabla f) \, ^T \, S \, M_A  ^{-1} \, (A \, \nabla f) \, d\mu .
\end{eqnarray*}
Finally we have
$$
\varepsilon \, \Vert g_\varepsilon \Vert _{L^2 (\mu)} \leq \varepsilon \, \int_0 ^\infty e^{-\varepsilon t} \, \Vert P_t f \Vert _{L^2 (\mu)} \, dt = \int_0 ^\infty e^{-s} \, \Vert P_{s/\varepsilon} f \Vert _{L^2 (\mu)} \, ds ,
$$
which converges to 0 as $\varepsilon \to 0$ by ergodicity and the dominated convergence theorem. The proof is achieved.
\end{proof}
Let us briefly mention an alternative proof of Theorem \ref{theo:BL} which avoid the intertwining approach emphasized in Theorem \ref{theo:intert}. More precisely, the argument is based on the $L^2$ method of H\"ormander and the technique developed by Bakry and his co-authors \cite{bakry_emery, BGL} to obtain functional inequalities such as Poincar\'e or log-Sobolev inequalities: the so-called $\Gamma_2$-calculus that we recall now the main idea. Since we have already defined the carr\'e du champ operator $\Gamma$, let us define the iterated operator $\Gamma_2$ by the following formula: for every $f,g\in \C ^\infty (\R^d , \R)$,
$$
\Gamma _2 (f,g) := \frac{1}{2} \, \left( \LL \Gamma (f,g) - \Gamma (f, \LL g) - \Gamma (\LL f, g)\right) .
$$
Then the famous result is the following: given a positive constant $\lambda$, the Poincar\'e inequality \eqref{eq:Poincare} is satisfied with constant $\lambda$ if and only if the inequality
$$
\int_{\R ^d} \Gamma _2 (f,f) \, d\mu \geq \lambda \, \int_{\R ^d} \Gamma  (f,f) \, d\mu ,
$$
holds for every $f\in \C ^\infty _0 (\R^d , \R)$, which rewrites by invariance of the measure $\mu$ and integration by parts as the inequality
$$
\int_{\R ^d} (\LL f )^2 \, d\mu \geq \lambda \, \int_{\R ^d} \vert \nabla f \vert ^2 \, d\mu .
$$
In our context the operator $\Gamma_2$ is given by
$$
\Gamma _2 (f,f) = \Vert \nabla \nabla f \Vert _{HS} ^2  + (\nabla f) \,  ^T \, \nabla \nabla V \, \nabla f ,
$$
the norm $\Vert \cdot \Vert _{HS}$ standing for the Hilbert-Schmidt norm of the matrix $\nabla \nabla f$. Therefore a sufficient condition ensuring the Poincar\'e inequality \eqref{eq:Poincare} is to assume that the potential $V$ is strongly convex, leading to the previously mentioned Bakry-\'Emery criterion. \vspace{0.1cm}

As in the $\Gamma_2$-calculus we start now by the quantity $\mu ((\LL f)^2 )$ and want to write it as the sum of two terms involving the distortion matrix $A$: a first term resembling to the operator $\Gamma$ plus a second term we hope to be non-negative. In other words, the game is to extract the minimum positivity of this expected non-negative term in order to get the second term as large as possible. We have for every $f\in \C ^\infty _0 (\R^d , \R)$,
\begin{eqnarray}
\label{eq:gamma_2A}
\nonumber \int_{\R ^d} (\LL f )^2 \, d\mu & = & \int_{\R ^d} \Vert A^{-1} \, \nabla (A \, \nabla f) \Vert _{HS} ^2 \, d\mu \\
\nonumber & & + \int_{\R ^d} (A\, \nabla f )^T \, S \, \left( A \, \nabla \nabla V \, A^{-1} - A \, \L A^{-1} \right) \, A \,  \nabla f \, d\mu \\
\nonumber & = & \int_{\R ^d} \Vert A^{-1} \, \nabla (A \, \nabla f) \Vert _{HS} ^2 \, d\mu + \int_{\R ^d} (\nabla f) \, ^T \, \left( \nabla \nabla V - \L A^{-1} \, A \right) \,  \nabla f \, d\mu \\
\nonumber & = & \int_{\R ^d} \Vert A^{-1} \, \nabla (A \, \nabla f) \Vert _{HS} ^2 \, d\mu + \int_{\R ^d} (\nabla f) \, ^T \, A^{-1} \, M_A \, A \,  \nabla f \, d\mu \\
& \geq & \int_{\R ^d} (\nabla f) \, ^T \, A^{-1} \, M_A \, A \,  \nabla f \, d\mu .
\end{eqnarray}
We are now able to give our second proof of Theorem \ref{theo:BL}.
\begin{Proof_th}
Since the previous approximation procedure can be adapted to the present proof, we assume to simplify the presentation that the operator $-\LL$ is bounded from below by some positive constant in the sense of self-adjoint operators on $L^2(\mu)$. Hence for every $f\in \C ^\infty _0 (\R^d , \R)$ such that $\mu (f) = 0$, the Poisson equation $-\LL g = f$ has a unique centered solution $g \in L^2 (\mu)$. Then by Cauchy-Schwarz' inequality,
\begin{eqnarray*}
\Var_\mu (f) & = & \int_{\R ^d} f^2 \, d\mu \\
& = & - \int_{\R ^d} f\, \LL g \, d\mu \\
& = & \int_{\R ^d} (\nabla f) ^T \, \nabla g \, d\mu \\
& \leq & \sqrt{\int_{\R ^d} (\nabla f) \, ^T \,  ( A^{-1} \, M_A \, A ) ^{-1} \, \nabla f \, d\mu } \, \sqrt{\int_{\R ^d} (\nabla g) \, ^T \,  A^{-1} \, M_A \, A \, \nabla g \, d\mu } \\
& \leq & \sqrt{\int_{\R ^d} (\nabla f) \, ^T \,  ( A^{-1} \, M_A \, A ) ^{-1} \, \nabla f \, d\mu } \, \sqrt{\int_{\R ^d} (\LL g )^2 \, d\mu } \\
& = & \sqrt{\int_{\R ^d} (\nabla f) \, ^T \,  ( A^{-1} \, M_A \, A ) ^{-1} \, \nabla f \, d\mu } \, \sqrt{\Var_\mu (f)} ,
\end{eqnarray*}
where we used \eqref{eq:gamma_2A} to obtain the last inequality. Finally dividing both sides by $\sqrt{\Var_\mu (f)}$ and squaring the inequality leads to the desired result.
\end{Proof_th}

As we have already seen, the matrix appearing in the right-hand-side of \eqref{eq:BLA} is
$$
A ^{-1} \, M_A \, A = \nabla \nabla V - \L A^{-1} \, A ,
$$
and therefore Theorem \ref{theo:BL} can be considered as an extension of the classical Brascamp-Lieb inequality covered by the choice of the distortion matrix $A = I$. Note however that for a given invertible matrix $A$ satisfying the assumptions of Theorem \ref{theo:BL}, the resulting generalized Brascamp-Lieb might not be directly comparable to the classical one. In particular there is no explicit expression of the possible extremal functions. \vspace{0.1cm}

Dealing now with the notion of spectral gap, an immediate application of Theorem \ref{theo:BL} entails the following result, which is an extension to the multi-dimensional setting of the famous variational formula of Chen and Wang established in the one-dimensional case \cite{chen_wang}.
\begin{theo}
\label{theo:spectral}
Assume that the matrix $A^{-1} \, M_A \, A$ is symmetric and its smallest eigenvalue $\rho_A$ is bounded from below by some positive constant. Then the spectral gap $\lambda_1 (-\LL , \mu)$ satisfies
\begin{equation}
\label{eq:rho_A}
\lambda_1 (-\LL , \mu) \geq \inf \, \rho_A .
\end{equation}
\end{theo}
In the one-dimensional case, the equality holds at least if $\lambda_1 (-\LL , \mu)$ is an eigenvalue of $-L$. However the optimality is not so clear in the multi-dimensional context. Indeed, to obtain the equality in \eqref{eq:rho_A}, one needs to get the equalities in the proof of Theorem \ref{theo:BL}, i.e. if the eigenvector $f_1$ associated to the spectral gap $\lambda_1 (-\LL , \mu)$ exists, the question is to find a good matrix $A$ satisfying the assumptions of Theorem \ref{theo:spectral} and such that
$$
\L_A  (A \, \nabla f_1 ) = 0 \quad \mbox{ and } \quad M_A \, A \, \nabla f_1 = \lambda_1 (-\LL , \mu) \, A \, \nabla f_1 .
$$
In particular one deduces that $\lambda_1 (-\LL , \mu)$ is also an eigenvalue of the matrix $M_A$ with associated eigenvector $A \, \nabla f_1 $. In contrast to the one-dimensional case \cite{chen_wang, bj}, for which we know that $f_1$ is strictly monotone and thus the optimal choice of function $a$ is $a = 1/f_1 '$, the multi-dimensional setting is more delicate since we have no idea of the behaviour of $f_1$ (except for a product measure $\mu$ for which we take for $A$ a diagonal matrix with the $1/(f_1 ^i) '$ on the diagonal). \vspace{0.1cm}

Let us continue to explore the consequences of the intertwining approach emphasized in Theorem \ref{theo:intert} in terms of spectral gap. In contrast to Theorem \ref{theo:spectral} where we exhibit a lower bound on the spectral gap given by the infimum of a certain quantity related to the matrix $A^{-1} \, M_A \, A$, we propose now an alternative lower bound which can be seen as an integrated version of the latter lower bound. These kind of results already appeared in a work of Veysseire \cite{veysseire} in the context of compact Riemannian manifolds and also in the recent article \cite{bj} for one-dimensional diffusions by means of the intertwining approach. On the basis of Theorem \ref{theo:intert}, we have in mind the presence of the weight matrix $A$ in the forthcoming formula. Our result is the following.
\begin{theo}
\label{theo:veysseire}
Assume that the matrix $A^{-1} \, M_A \, A$ is symmetric and its smallest eigenvalue $\rho_A$ is bounded from below by some positive constant. Assume also that the matrix $S$ is uniformly bounded from below by $\alpha$ and from above by $\beta$, where $\alpha, \beta $ are some positive constants. Then we have the lower bound on the spectral gap
$$
 \lambda_1 (-\LL , \mu) \geq \frac{1}{  \left(\int _{\R ^d}\frac{d\mu}{\rho_A} \right) + \frac{\left(1-\frac{\alpha}{\beta} \right)}{\inf \rho_A}}.
$$
\end{theo}
\begin{Proof}
We start with the variance identity established in the proof of Theorem \ref{theo:BL}: for every $f\in \C_0 ^\infty (\R^d , \R)$,
\begin{eqnarray*}
\Var_\mu(f) & = & \int_0 ^\infty \int_{\R^d} (A \, \nabla f) \, ^T \, S \, \Q_{t,A} ^{M_A} (A \, \nabla f) \, d\mu \, dt ,
\end{eqnarray*}
which leads by Cauchy-Schwarz' inequality to the inequality
\begin{eqnarray}
\label{eq:var_ineq}
\nonumber \Var_\mu(f) & \leq & \int_0 ^\infty \sqrt{\int_{\R^d} (A \, \nabla f) \, ^T \, S \, A \, \nabla f \, d\mu} \, \sqrt{\int_{\R^d} (\Q_{t,A} ^{M_A} (A \, \nabla f) ) ^T \, S \, \Q_{t,A} ^{M_A} (A \, \nabla f) \, d\mu} \, dt \\
& = & \int_0 ^\infty \sqrt{\int_{\R^d} \vert \nabla f \vert ^2 \, d\mu} \, \sqrt{\int_{\R^d} (\Q_{t,A} ^{M_A} (A \, \nabla f)) ^T \, S \, \Q_{t,A} ^{M_A} (A \, \nabla f) \, d\mu} \, dt .
\end{eqnarray}
Denote
$$
\Lambda (\L_A ^{M_A}) := \inf \left \{ - \int_{\R ^d} F^T \, S \, \L _A ^{M_A} F \, d\mu : F \in \D (\L_A ^{M_A}); \, \int_{\R ^d} F^T \, S \, F \, d\mu = 1 \right \} ,
$$
the bottom of the spectrum in $L^2(S,\mu)$ of the Schr\"odinger operator $-\L_A ^{M_A}$. Recall that we have
$$
- \int_{\R ^d} F^T  \, S \, \L _A ^{M_A} F \, d\mu = \int_{\R ^d} (\nabla F) \, ^T \, S \, \nabla F \, d\mu + \int_{\R ^d} F^T \, S \, M_A \, F \, d\mu ,
$$
hence by our first assumption we already know that $\Lambda (\L_A ^{M_A}) \geq \inf\, \rho_A >0$. Since the operator $\L_A ^{M_A} + \Lambda (\L_A ^{M_A}) \, I $ with domain $\D (\L_A ^{M_A})$ is dissipative on $L^2(S,\mu)$, the semigroup $(e^{\Lambda (\L_A ^{M_A}) t } \, \Q_{t,A} ^{M_A})_{t\geq 0} $ is a contraction semigroup on $L^2(S,\mu)$ and thus we have the estimate
\begin{eqnarray*}
\int_{\R^d} (\Q_{t,A} ^{M_A} (A \, \nabla f) )^T \, S \, \Q_{t,A} ^{M_A} (A \, \nabla f) \, d\mu & \leq & e^{-2 \, \Lambda (\L_A ^{M_A}) \, t } \, \int_{\R^d} (A \, \nabla f) \, ^T \, S \, A \, \nabla f \, d\mu \\
& = & e^{-2 \, \Lambda (\L_A ^{M_A}) \, t } \, \int_{\R^d} \vert \nabla f \vert ^2 \, d\mu .
\end{eqnarray*}
Plugging then in the variance inequality \eqref{eq:var_ineq} entails the Poincar\'e inequality
$$
\Var_\mu (f) \leq \frac{1}{\Lambda (\L_A ^{M_A})} \, \int_{\R^d} \vert \nabla f \vert ^2 \, d\mu ,
$$
i.e. we have the comparison
\begin{equation}
\label{eq:lambda1_MA}
\lambda _1 (-\LL ,\mu) \geq \Lambda (\L_A ^{M_A}).
\end{equation}
Now we aim at bounding from below the quantity $\Lambda (\L_A ^{M_A})$ by a constant depending on $\lambda _1 (-\LL ,\mu)$. To simplify the notation in the sequel of the proof, we denote $\lambda_1$ the spectral gap $\lambda _1 (-\LL ,\mu)$. Let $F$ be a smooth vector field and denoting $h = \sqrt{F^T \, S \, F}$ we assume that $\int_{\R ^d} h^2 \, d\mu = 1 $. On the one hand we have
$$
\int_{\R ^d} F^T \, S \, M_A \, F \, d\mu \geq \int_{\R ^d} \rho_A \, h^2 \, d\mu ,
$$
and on the other hand the assumption on the matrix $S$ together with the Poincar\'e inequality entail the following computations:
\begin{eqnarray*}
\int_{\R ^d} (\nabla F) \, ^T \, S \, \nabla F \, d\mu & \geq & \alpha \, \int_{\R ^d} (\nabla F) \, ^T \, \nabla F \, d\mu \\
& = & \alpha \, \sum_{i=1} ^d \int_{\R ^d} \vert \nabla F_i \vert ^2 \, d\mu \\
& \geq & \alpha \, \lambda_1 \, \sum_{i=1} ^d \Var _\mu (F_i) \\
& = & \alpha \, \lambda_1 \, \left( \int_{\R ^d} \vert F\vert ^2 \, d\mu - \left \vert \int _{\R ^d} F \, d\mu \right \vert ^2 \right) \\
& \geq & \frac{\alpha}{\beta} \, \lambda_1 \, \int_{\R ^d} h^2 \, d\mu - \lambda_1 \, \left( \int _{\R ^d} F \, d\mu\right) ^T \, S \, \left( \int _{\R ^d} F \, d\mu\right) \\
& \geq & \frac{\alpha}{\beta} \, \lambda_1 - \lambda_1 \, \left( \int _{\R ^d} h \, d\mu\right) ^2 ,
\end{eqnarray*}
where to obtain the last inequality we used Jensen's inequality. Therefore coming back to the definition of $\Lambda (\L_A ^{M_A})$ and using Cauchy-Schwarz' inequality, we get
\begin{eqnarray*}
\Lambda (\L_A ^{M_A}) & \geq & \frac{\alpha}{\beta} \, \lambda_1 + \inf \left \{ \int_{\R ^d} \rho_A \, h^2 \, d\mu  - \lambda_1 \, \left( \int _{\R ^d} h \, d\mu\right) ^2 : h \in L^2 (\rho_A \, d\mu); \, \int _{\R ^d} h^2  \, d\mu = 1 \right \} \\
& \geq & \frac{\alpha}{\beta} \, \lambda_1 + \inf \left \{ \int_{\R ^d} \rho_A \, h^2 \, d\mu  \, \left( 1 - \lambda_1 \, \int_{\R ^d} \frac{d\mu}{\rho_A} \right) : h \in L^2 (\rho_A \, d\mu); \, \int _{\R ^d} h^2  \, d\mu = 1 \right \},
\end{eqnarray*}
and combining with \eqref{eq:lambda1_MA} yields to
\begin{equation}
\label{eq:lambda_inf}
\left( 1-\frac{\alpha}{\beta} \right) \, \lambda_1 \geq \inf \left \{ \int_{\R ^d} \rho_A \, h^2 \, d\mu  \, \left( 1 - \lambda_1 \, \int_{\R ^d} \frac{d\mu}{\rho_A} \right) : h \in L^2 (\rho_A \, d\mu); \, \int _{\R ^d} h^2  \, d\mu = 1 \right \}.
\end{equation}
Now we observe that although two different cases may occur, both lead to the desired conclusion. Indeed if
$$
1 - \lambda_1 \, \int_{\R ^d} \frac{d\mu}{\rho_A} < 0 ,
$$
then the conclusion trivially holds whereas if
$$
1 - \lambda_1 \, \int_{\R ^d} \frac{d\mu}{\rho_A} \geq 0,
$$
then \eqref{eq:lambda_inf} entails the inequality
$$
\left( 1-\frac{\alpha}{\beta} \right) \, \lambda_1 \geq \inf \, \rho_A  \, \left( 1 - \lambda_1 \, \int_{\R ^d} \frac{d\mu}{\rho_A} \right) ,
$$
and rearranging the terms completes the proof.
\end{Proof}
As announced, our result generalizes that of Veysseire \cite{veysseire} in the sense that the choice of the identity matrix $I$ for $A$ entails the inequality
$$
\lambda_1 (-\LL , \mu) \geq \frac{1}{\int_{\R ^d} \frac{d\mu}{\rho_1}} ,
$$
where the smallest eigenvalue $\rho_1$ of the matrix $\nabla \nabla V$ is assumed to be positive. In particular in this case we might avoid the assumption $\inf \rho_1 >0$ and replace it by $\rho_1 \geq 0$ by using an approximation procedure (certainly in this situation it may happen that the above integral is infinite, hence giving no information on the spectral gap). Let us give the short argument for the sake of completeness. We only assume that $\rho_1 \geq 0$, i.e. $\nabla \nabla V$ is a positive semi-definite matrix and thus the measure $\mu$ is log-concave.
Applying Theorem \ref{theo:veysseire} with $A = I$ to the generator $\LL ^\varepsilon := \Delta - (\nabla V_\varepsilon)^T \, \nabla$ associated to the strongly convex potential $V_\varepsilon := V+ \varepsilon \vert \cdot \vert ^2 /2$ and whose invariant measure $\mu_\varepsilon$ has Lebesgue density proportional to $e^{-V _\varepsilon}$, we have
\begin{eqnarray*}
\lambda_1 (- \LL ^\varepsilon , \mu_\varepsilon) & \geq & \frac{1}{\int_{\R ^d} \frac{d\mu_\varepsilon }{\rho_1 + \varepsilon}} .
\end{eqnarray*}
In other words the generator $L ^\varepsilon$ corresponds to the approximation already used in the proof of Theorem \ref{theo:BL} with $A = I$ since it is straightforward to observe the following intertwining relation:
$$
\nabla L ^\varepsilon f = (\L - \nabla \nabla V_\varepsilon ) \, (\nabla f) = (\L - \nabla \nabla V - \varepsilon \, I ) \, (\nabla f).
$$
By Beppo Levi's theorem, the integral in the right-hand-side of the inequality above tends to $\int_{\R ^d} (1/\rho_1) \, d\mu$ as $\eps \to 0$. Now let $\eta>0$ and $f_\eta \in \D (\cE _\mu)$ be such that
\[
\frac{ \int _{\R ^d} \left \vert \nabla f_\eta \right \vert ^2 \, d\mu }{ \int_{\R^d} f_{\eta} ^2 \, d\mu - \left( \int_{\R^d} f_\eta \, d\mu \right)^2  } \leq \lambda_1 (-\LL , \mu ) + \eta.
\]
Since $\D (\cE _\mu) \subset \D (\cE _{\mu_\eps})$ we have
\[
 \lambda_1 (-\LL ^\eps , \mu_\eps ) \leq \frac{ \int _{\R^d} \left \vert \nabla f_\eta \right \vert ^2 \, d\mu _\eps }{ \int_{\R^d} f_\eta^2 \, d\mu_\eps - \left(   \int_{\R^d} f_\eta \, d\mu_\eps \right)^2  },
\]
and by Beppo Levi's theorem together with the dominated convergence theorem, we get at the limit $\eps \to 0$:
\[
 \limsup_{\eps \to 0} \lambda_1 (-\LL ^\eps , \mu_\eps ) \leq \frac{ \int _{\R ^d} \left \vert \nabla f_\eta \right \vert ^2 \, d\mu }{ \int_{\R^d} f_{\eta} ^2 \, d\mu - \left( \int_{\R^d} f_\eta \, d\mu \right)^2  } \leq \lambda_1 + \eta.
\]
Finally letting $\eta \to 0$ gives the desired conclusion.

\section{Asymmetric Brascamp-Lieb inequalities}
\label{sect:asym_BL}
Our main result Theorem \ref{theo:intert} also allows us to obtain Brascamp-Lieb type inequalities with the covariance instead of the variance, in the spirit of the works of Helffer \cite{helffer} and Ledoux \cite{ledoux} about decay of correlations for spin systems, see also the recent articles \cite{menz_otto, carlen_cordero, menz}. Such inequalities are called asymmetric Brascamp-Lieb inequalities since the two functions are decorrelated in the sense that the right-hand-side of the inequality is the product of two conjugate $L^p$ norms of their weighted gradients. Note that such covariance estimates are also useful to derive concentration results, as regards the papers of Houdr\'e and his co-authors \cite{bht, houdre}. \vspace{0.1cm}

In order to establish these asymmetric Brascamp-Lieb inequalities, we need some material and in particular an additional ingredient which is the stochastic representation of our Feynman-Kac type semigroups of interest. Actually one can show that such a representation holds when the invertible matrix $A$ is a multiple of the identity (we ignore if this is true in the general case). For the choice of a distortion matrix $A = a \, I$, where $a$ is a smooth positive function on $\R^d$, we denote in the sequel $M_a$, $\L _a ^{M_a}$, $\L _a$ and $(\Q_{t,a} ^{M_a})_{t\geq 0}$ the corresponding operators and semigroup acting on vector fields. The key point in the forthcoming analysis resides in the fact that since $\L_a$ is a diagonal operator, it can also be interpreted as an operator acting on functions. In this case we denote $\LL_a$ the corresponding diffusion operator: for every $f\in \C ^\infty (\R^d , \R)$,
\begin{eqnarray*}
L_a f & = & \LL f + 2 \, a \, (\nabla a^{-1}) ^T \, \nabla f \\
& = & \Delta f - (\nabla V_a) ^T \, \nabla f,
\end{eqnarray*}
where $V_a$ is the smooth potential $V_a = V + \log (a^2)$. This operator is symmetric and non-positive on $\C _0 ^\infty (\R^d , \R)$ with respect to the measure $\mu_a = a^{-2} \cdot \mu$, hence it is essentially self-adjoint on $\C _0 ^\infty (\R^d , \R)$ and admits a unique self-adjoint extension (still denoted $L_a $) with domain $\D (L_a) \subset L^2 (\mu_a)$. Notice that the measure $\mu_a$ has no reason to be finite \textit{a priori} - such an assumption is not required in the sequel - and moreover the corresponding process might explode, i.e. it goes to infinity in finite time or, equivalently, the associated semigroup $(P_{t,a})_{t\geq 0}$ is not stochastically complete: $P_{t,a} 1 \leq 1$ for some (hence for all) $t >0$. In particular, an analytical sufficient and ``easy-to-check" condition ensuring the non-explosion is to assume that the Hessian matrix of the potential $V_a$ is uniformly bounded from below, cf. \cite{bakry:non-exp}. An alternative usual criterion is the so-called Lyapunov (or drift) condition \textit{\`a la Meyn-Tweedie}, cf. \cite{meyn_tweedie}, stating the existence of a smooth positive function $f$ going to infinity at infinity and two constants $\alpha , \beta \geq 0$ such that
$$
L_a f \leq \alpha \, f +\beta .
$$
In the sequel we denote $(X_{t,a} ^x)_{t\geq 0}$ the (potentially minimal) diffusion process starting from $x\in \R ^d$ and whose generator is $L_a$. \vspace{0.1cm}

Now we are ready to state the main result of the present section. Note that since we choose $A = a\, I$, the matrix $A^{-1} \, M_A \, A$ is automatically symmetric and we have
$$
A^{-1} \, M_A \, A = M_a = \nabla \nabla V - (a \, \LL a^{-1}) \, I .
$$
Therefore to invoke Theorem \ref{theo:intert} in the proof below one needs only to assume that the smallest eigenvalue of $M_a$ is bounded from below. Denote the covariance under $\mu$ of two given functions $f,g \in L^2 (\mu)$ as
$$
\Cov _\mu (f,g) := \mu (f\, g) - \mu (f) \, \mu (g) .
$$
Our result stands as follows.
\begin{theo}
\label{thm:asym-BL}
Assume that the smallest eigenvalue $\rho_a$ of the matrix $M_a$ is positive. Moreover assume either the Hessian matrix of the potential $V_a$ is uniformly bounded from below or the Lyapunov condition written above. Then for every functions $f,g \in \C _0 ^\infty (\R ^d , \R)$, we have
\begin{equation}
\label{eq:asym_BL}
\left \vert \Cov _\mu (f,g) \right \vert \leq \left \Vert \frac{ | a \, \nabla g|}{\rho_a } \right\Vert _\infty \, \int_{\R ^d} \frac{\vert \nabla f \vert}{a} \, d\mu .
\end{equation}
\end{theo}
If we choose $a \equiv 1$ then we have $M_a = \nabla \nabla V $ and all the assumptions reduce to the positivity of the smallest eigenvalue $\rho_1$ of the matrix $\nabla \nabla V$, so that the inequality \eqref{eq:asym_BL} becomes
$$
\left \vert \Cov _\mu (f,g) \right \vert \leq \left \Vert \frac{ |\nabla g|}{\rho_1} \right\Vert _\infty \, \int_{\R ^d} \vert \nabla f \vert \, d\mu .
$$
In other words such an inequality is nothing but the case $(p,q) = (\infty, 1)$ of the multi-dimensional version appearing in \cite{carlen_cordero} of the original one-dimensional result established by Menz and Otto in \cite{menz_otto}. Once again our main contribution comes from the freedom in the choice of the distortion function $a$. \vspace{0.1cm}

Let us provide a guideline strategy of the proof of Theorem \ref{thm:asym-BL}. Starting by the covariance version of the variance identity \eqref{eq:var_id}, i.e for every $f,g \in \C _0 ^\infty (\R^d , \R)$,
\begin{eqnarray*}
\Cov _\mu (f,g) & = & \int_{\R^d} (a \, \nabla f) ^T \, (- \L _a ^{M_a})^{-1} (a \, \nabla g) \, d\mu _a ,
\end{eqnarray*}
we are done once we have proven the following regularization result: for every $g\in \C _0 ^\infty (\R ^d , \R)$,
$$
\left \Vert \vert (-\L _a ^{M_a})^{-1} (a\nabla g)\vert \right \Vert _\infty \leq \left \Vert \frac{\vert a\nabla g \vert }{\rho_a} \right \Vert _\infty ,
$$
and through the representation
\[
(-\L _a ^{M_a})^{-1} (a\nabla g) =\int_0^{\infty} \Q_{t,a}^{M_a} (a\nabla g) \, dt,
\]
it means that a $L^\infty$ bound on the underlying semigroup $(Q_{t,a} ^{M_a})_{t\geq 0}$ acting on weighted gradients is required. However, in contrast to the classical Feynman-Kac situation where we focus our attention on a Feynman-Kac semigroup acting on functions, the semigroup $(Q_{t,a} ^{M_a})_{t\geq 0}$ has no reason to be bounded \textit{a priori} and that is why a careful attention has to be brought to this problem. Actually the intertwining relation of Theorem \ref{theo:intert} entails many interesting consequences and among them the potential boundedness of $Q_{t,a} ^{M_a} (a\nabla f)$ reduces to that of $a \nabla P_t f$. Hence the first part of this section is devoted to show this boundedness property, leading to the result below. We admit that the proof is somewhat technical but the result is interesting in its own right.

\begin{theo}
\label{thm:compar_feynman}
Assume that the smallest eigenvalue $\rho_a$ of the matrix $M_a$ is bounded from below by some real constant $k_a \in \R$. Letting $f \in  \C ^\infty _0 (\R^d, \R)$, we have the inequality
\begin{equation}
\label{eq:sub}
a \, \vert \nabla P_t f \vert \leq e^{- k_a t } \, P_{t,a} (\vert a\nabla f \vert ), \quad t \geq 0.
\end{equation}
\end{theo}
In the case $a \equiv 1$ we recover the classical Bakry-\'Emery criterion, cf. for instance \cite{bakry:non-exp}, asserting that a uniform lower bound $\lambda$ on the Hessian of the potential $V$ ensures the sub-commutation relation
$$
\vert \nabla P_t f \vert \leq e^{-\lambda t} \, P_t (\vert \nabla f \vert ), \quad t\geq 0.
$$
To prove Theorem \ref{thm:compar_feynman}, we will follow the approach emphasized in \cite{bakry:non-exp} and propose an analytic proof divided into several steps. Recall that the norm induced by $L$ and $L_a$ are respectively given by
$$
\Vert f\Vert _{\D (L)} ^2 =  \Vert f \Vert ^2 _{L^2 (\mu)} + \Vert L f \Vert ^2 _{L^2 (\mu)},
$$
and
$$
\Vert f\Vert _{\D (L_a)} ^2 =  \Vert f \Vert ^2 _{L^2 (\mu_a)} + \Vert L_a f \Vert ^2 _{L^2 (\mu_a)}.
$$

\begin{lemme}
\label{lem1:non-exp}
Assume that the smallest eigenvalue $\rho_a$ of the matrix $M_a$ is bounded from below by some real constant $k_a \in \R$. Then for every smooth function $f\in \D (L)$, we have the inequality
\begin{equation}
\label{eq:DL}
\left \Vert \, \Vert \nabla (a\nabla f) \Vert _{HS} \, \right \Vert_{L^2(\mu_a)}^2\leq \left( 1+ \frac{k_a ^-}{2} \right) \Vert f \Vert_{\D(\LL)}^2 ,
\end{equation}
where $k_a ^- := - \min \{ 0, k_a \}$.
\end{lemme}
\begin{proof}
First, let us start by proving the result for every function $f\in\C ^\infty _0 (\R^d, \R)$. We have
\begin{align*}
\label{eq:gamma_2A}
 \int_{\R ^d} (\LL f )^2 \, d\mu
 & = & \int_{\R ^d} \Vert a^{-1} \, \nabla (a \, \nabla f) \Vert _{HS} ^2 \, d\mu
 &+ \int_{\R ^d} (\nabla f) ^T \,  \left(  \nabla \nabla V \, -  a\, \LL a^{-1} \, I \right) \, \nabla f \, d\mu \\
 & = & \int_{\R ^d} \Vert  \nabla (a \, \nabla f) \Vert _{HS} ^2 \, d\mu_a
 &+ \int_{\R ^d} (\nabla f) ^T \,  \left(  \nabla \nabla V \, -  a\, \LL a^{-1} \, I \right) \, \nabla f \, d\mu \\
 &\geq & \int_{\R ^d} \Vert \nabla (a \, \nabla f) \Vert _{HS} ^2 \, d\mu_a
 &+ k_a \int_{\R ^d} \vert \nabla f \vert ^2 \, d\mu.
\end{align*}
Since we have
\[
\int_{\R ^d} \vert \nabla f \vert ^2 \, d\mu = - \int_{\R ^d}  f  \,  \LL  f \, d\mu \leq \frac{1}{2} \int_{\R^d} f^2 \, d\mu + \frac{1}{2}  \int_{\R^d} (\LL f)^2 \,  d\mu,
\]
the conclusion follows at least when $f\in \C ^\infty _0 (\R^d , \R)$. \vspace{0.1cm} \\
Now let us extend the desired inequality for general smooth functions in $\D (L)$. To do so, given a smooth function $f\in \D (L)$, let $(f_n)_{n\in \N}$ be a sequence of functions in $\C ^\infty _0 (\R^d, \R)$ converging to $f$ for the norm induced by $L$ (the essential self-adjointness property allows the existence of such a sequence). Then it is a Cauchy sequence in $L^2 (\mu)$ and by \eqref{eq:DL} the sequence $(\Vert \nabla (a\nabla f_n) \Vert _{HS})_{n\in \N}$ reveals also to be Cauchy in $L^2 (\mu_a)$, hence converges in $L^2 (\mu_a)$ to (the Hilbert-Schmidt norm of) some matrix-valued function $K$. On the one hand if the limiting function $f$ lies in $\C ^\infty _0 (\R^d , \R)$ then one sees easily that $K= \nabla (a\nabla f)$. On the other hand, if the support of $f$ is not compact, then we proceed as follows. Fix some $g \in\C ^\infty _0 (\R^d, \R)$ and consider the product $g f_n \in \C _0 ^\infty (\R^d , \R)$. Because of the rule
$$
L (g (f_n - f)) = (f_n - f) \, L g + g \, L (f_n -f) + 2 \, (\nabla g) ^T \, \nabla (f_n -f),
$$
we observe immediately that the sequence $ (g f_n)_{n\in \N}$ converges to $g f$ with respect to the norm induced by $L$ and since $gf$ has compact support, one deduces that $\Vert \nabla (a\nabla (gf_n)) \Vert _{HS}$ converges in $L^2 (\mu_a)$ to $\Vert \nabla (a\nabla (gf)) \Vert _{HS}$. Now we have
 \[
\nabla (a\nabla (gf_n)) =  \nabla (a \nabla g) \, f_n +  a \, \nabla f_n \, (\nabla g)^T + a \, \nabla g \, (\nabla f_n) ^T +  \nabla (a \nabla f_n) \, g,
 \]
and since all the above quantities converge (more precisely all the Hilbert-Schmidt norms converge in $L^2 (\mu_a)$), we obtain at the limit
\[
\nabla (a\nabla (gf)) =  \nabla (a \nabla g) \, f +  a \, \nabla f \, (\nabla g)^T + a \, \nabla g \, (\nabla f) ^T +  K \, g.
 \]
Finally, the uniqueness of the limit and the fact that the equality above holds for every function $g \in\C ^\infty _0 (\R^d, \R)$, one infers that $K = \nabla (a \nabla f)$ and thus \eqref{eq:DL} immediately holds in full generality. The proof is now achieved.
\end{proof}
The reason why the result in Lemma \ref{lem1:non-exp} has to be extended to smooth functions in $\D (L)$ is that we will apply in a moment such a result for $f = P_t g$, which is a smooth function by ellipticity of the semigroup but not compactly supported, even if $g$ is. \vspace{0.1cm}

The second preliminary result required to establish Theorem \ref{thm:compar_feynman} is the following lemma.
\begin{lemme}
\label{lem2:non-exp}
Assume that the smallest eigenvalue $\rho_a$ of the matrix $M_a$ is bounded from below by some real constant $k_a \in \R$. Let $f\in\C ^\infty _0 (\R^d, \R)$ and $ t > 0$ be a finite time horizon. For $0\leq s\leq t$ and $\ve >0$, the smooth function
$$
f_s ^\ve := \sqrt{e^{- 2 k_a s} \, \vert a \, \nabla P_{t-s}f \vert ^2 + \ve^2 } -\ve ,
$$
satisfies the inequality
$$
\LL_a f_s ^\ve + \partial _s f_s ^\ve \geq 0 , \quad 0 \leq s \leq t .
$$
\end{lemme}

\begin{proof}
Using the intertwining relation at the level of the generators, that is, for every $g\in C^\infty (\R^d , \R)$,
$$
a \, \nabla \LL g = (\L_a - M_a) (a \, \nabla g) ,
$$
we have
\begin{eqnarray*}
\partial _s f_s ^\ve & = & \frac{e^{- 2 k_a s} \left(-k_a \, |a\nabla P_{t-s}f|^2- (a \, \nabla P_{t-s}f) ^T \, a \, \nabla \LL P_{t-s}f \right)}{ \sqrt{e^{- 2 k_a s} \, \vert a \, \nabla P_{t-s}f \vert ^2 + \ve^2}} \\
& \geq & \frac{e^{- 2 k_a s} \,  \left( (\rho_a-k_a) \, \vert a \, \nabla P_{t-s}f \vert ^2-  (a \, \nabla P_{t-s}f) ^T \, \L_a (a \, \nabla  P_{t-s}f)\right) }{ \sqrt{e^{- 2 k_a s} \, \vert a \, \nabla P_{t-s}f \vert ^2 + \ve^2}} ,
\end{eqnarray*}
and using the diffusion property,
$$
L_a (\sqrt u) = \frac{1}{2  \sqrt u } \, L_a u - \frac{1}{ 4  u^{3/2} } \, \vert \nabla u \vert ^2 ,
$$
one gets
\begin{align*}
L_a f_s ^\ve =
&\frac{e^{- 2 k_a s} \, L_a ( \vert a \, \nabla P_{t-s}f \vert ^2)}{ 2 \, \sqrt{e^{- 2 k_a s} \, \vert a \, \nabla P_{t-s}f \vert ^2 + \ve^2 }}
- \frac{e^{- 4 k_a s} \, \vert a \, \nabla P_{t-s}f \vert ^2 \, \Vert \nabla (a\nabla P_{t-s}f) \Vert _{HS} ^2}{  \left( e^{- 2 k_a s} \, \vert a \, \nabla P_{t-s}f \vert ^2 + \ve^2\right)^\frac{3}{2}} .\\
\end{align*}
Since
\[
L_a ( \vert a\, \nabla P_{t-s}f \vert ^2 ) = 2 \, (a \, \nabla P_{t-s}f) ^T \, \L_a (a \, \nabla P_{t-s}f) + 2 \, \Vert \nabla (a\nabla P_{t-s}f) \Vert _{HS} ^2 ,
\]
one finds
\[
L_a f_s ^\ve = \frac{e^{- 2 k_a s} \, (a \, \nabla P_{t-s}f) ^T \, \L_a (a\, \nabla P_{t-s}f)}{\sqrt{e^{- 2 k_a s} \, \vert a \, \nabla P_{t-s}f \vert ^2 + \ve^2}} + \frac{\ve ^2 \, e^{- 2 k_a s} \, \Vert \nabla (a\nabla P_{t-s}f) \Vert _{HS} ^2 }{\left( e^{- 2 k_a s} \, \vert a \, \nabla P_{t-s}f \vert ^2 + \ve^2\right) ^{3/2}} ,
\]
and therefore combining the two expressions above entail the inequality
\[
L_a f_s ^\ve + \partial _s f_s ^\ve \geq \frac{e^{- 2 k_a s} \, (\rho_a-k_a) \, \vert a \, \nabla P_{t-s}f \vert ^2} {\sqrt{e^{- 2 k_a s} \, \vert a \, \nabla P_{t-s}f \vert ^2 + \ve^2}} ,
\]
and since $\rho_a \geq k_a$ the desired conclusion follows.
\end{proof}

Once Lemmas \ref{lem1:non-exp} and \ref{lem2:non-exp} are established, we are now able to provide a detailed proof of Theorem \ref{thm:compar_feynman}.
\begin{proof}[Proof of Theorem \ref{thm:compar_feynman}]
Letting two non-negative functions $g_1 , g_2 \in \C _0 ^\infty (\R^d , \R)$, we fix a finite time horizon $t >0$ and consider the function of time
\[
\psi ^\ve (s) := \int_{\R^d} g_1 \, P_{s,a} (g_2 \, f_s ^\ve) \, d\mu_a  = \int_{\R^d} P_{s,a} g_1 \, g_2 \, f_s ^\ve \, d\mu_a , \quad 0 \leq s \leq t,
\]
where $f^\ve$ is the function defined in Lemma \ref{lem2:non-exp}. Since $g_1 , g_2$ are non-negative, one has by Lemma \ref{lem2:non-exp},
\begin{eqnarray*}
\partial _s \psi ^\ve (s) & = & \int_{\R^d} L_a P_{s,a} g_1 \, g_2 \, f_s ^\ve \, d\mu_a + \int_{\R^d} P_{s,a} g_1 \, g_2 \, \partial _s f_s ^\ve \, d\mu_a \\
& \geq & \int_{\R^d} L_a P_{s,a} g_1 \, g_2 \, f_s ^\ve \, d\mu_a - \int_{\R^d} P_{s,a} g_1 \, g_2 \, L_a f_s ^\ve \, d\mu_a \\
& = & - \int_{\R^d} (\nabla P_{s,a} g_1 )^T \, \nabla (g_2 \, f_s ^\ve ) \, d\mu_a + \int_{\R^d} (\nabla (P_{s,a} g_1 \, g_2) ) ^T \, \nabla f_s ^\ve \, d\mu_a \\
& = & - \int_{\R^d} (\nabla P_{s,a} g_1) ^T \, \nabla g_2 \, f_s ^\ve \, d\mu_a +  \int_{\R^d} P_{s,a} g_1 \, (\nabla g_2) ^T \, \nabla f_s ^\ve \, d\mu_a ,
\end{eqnarray*}
leading by Cauchy-Schwarz' inequality to the estimate:
\begin{eqnarray*}
\partial _s \psi ^\ve (s) & \geq & - \Vert \vert \nabla g_2 \vert \Vert_\infty  \left( \Vert \vert \nabla f_s ^\ve \vert \Vert_{L^2(\mu_a) } \, \Vert P_{s,a} g_1 \Vert_{L^2(\mu_a)} +  \Vert f_s ^\ve \Vert_{L^2(\mu_a) } \, \Vert \vert \nabla P_{s,a} g_1 \vert \Vert_{L^2(\mu_a)} \right) \\
& \geq & - \Vert \vert \nabla g_2 \vert \Vert_\infty  \left( \Vert \vert \nabla f_s ^\ve \vert \Vert_{L^2(\mu_a) } \, \Vert g_1 \Vert_{L^2(\mu_a)} +  \Vert f_s ^\ve \Vert_{L^2(\mu_a) } \, \Vert \vert \nabla g_1 \vert \Vert_{L^2(\mu_a)} \right) .
\end{eqnarray*}
To obtain the last inequality we used the contractivity of the semigroup in $L^2(\mu_a)$ and the fact that the function
$$
s\mapsto \int_{\R^d} \vert \nabla P_{s,a} g_1 \vert ^2 \, d\mu_a = - \int_{\R^d} P_{s,a} g_1 \, L_a P_{s,a} g_1 \, d\mu_a ,
$$
is non-increasing. Now, since we have $0 \leq f_s ^\ve \leq e^{- k_a s} \, \vert a \, \nabla P_{t-s} f \vert $, one gets
\begin{eqnarray*}
\int_{\R^d} {f_s ^\ve} ^2 \, d\mu_a & \leq & e^{-2 k_a s} \, \int_{\R^d} \vert a \, \nabla P_{t-s}f \vert ^2 \, d\mu_a \\
& = & e^{-2 k_a s} \, \int_{\R^d} \vert \nabla P_{t-s}f \vert ^2 \, d\mu \\
& \leq & e^{-2 k_a s} \, \int_{\R^d} \vert \nabla f \vert ^2 \, d\mu .
\end{eqnarray*}
Moreover we have
 \[
 \nabla f_s ^\ve = \frac{e^{- 2 k_a s} \, (a\, \nabla P_{t-s}f) ^T \, \nabla (a \, \nabla P_{t-s}f)}{\sqrt{e^{- 2 k_a s} \, \vert a \, \nabla P_{t-s}f \vert ^2 + \ve^2}},
 \]
and by Cauchy-Schwarz' inequality, we obtain
\[
\vert \nabla f_s ^\ve \vert \leq e^{-k_a s} \, \Vert \nabla (a \, \nabla P_{t-s}f) \Vert _{HS}.
\]
The key point then is to apply Lemma \ref{lem1:non-exp} to the semigroup $P_{t-s} f$, which is a smooth element of $\D (L)$. We thus get
\begin{eqnarray*}
\Vert \vert \nabla f_s ^\ve \vert \Vert_{L^2(\mu_a)} ^2 & \leq & e^{-2 k_a s} \, \left( 1+ \frac{k_a ^-}{2} \right) \Vert P_{t-s} f \Vert_{\D(\LL)}^2 \\
& \leq & e^{-2 k_a s} \, \left( 1+ \frac{k_a ^-}{2} \right) \Vert f \Vert_{\D(\LL)}^2 .
\end{eqnarray*}
Combining the estimates above entails the existence of a smooth and positive function of time $c$, depending only on $k_a$, such that we have
\[
\partial _s \psi ^\ve (s) \geq - c(s) \, \Vert \vert \nabla g_2 \vert \Vert_\infty \, \Vert f\Vert _{\D(L)} \, \Vert g_1\Vert _{\D(L_a)} .
\]
Integrating the above inequality between $0$ and $t$ and taking for $g_2$ an element $\phi_n \in \C ^\infty _0 (\R ^d , \R)$ valued between 0 and 1 and such that $\phi_n \uparrow 1$ pointwise and $\Vert \vert \nabla \phi_n \vert \Vert_\infty \to 0$ as $n\to \infty$ (recall that the existence of such a sequence is provided by the completeness property of the Euclidean space), we obtain at the limit $n \to \infty$:
\[
\int_{\R^d} g_1 \, \left( P_{t,a} f_t ^\ve - f_0 ^\ve \right) \, d\mu_a \geq 0.
\]
Finally, since the inequality above holds for every $g_1 \in \C_0 ^\infty (\R^d, \R)$ we get $P_{t,a} f_t ^\ve \geq f_0 ^\ve$ and letting $\ve \to 0$ achieves the proof of Theorem \ref{thm:compar_feynman}.
\end{proof}
A consequence of Theorem \ref{thm:compar_feynman} is the boundedness of the quantity $a \, \nabla P_t f$ as soon as $f\in \C_0 ^\infty (\R^d , \R)$, a property required to prove Theorem \ref{thm:asym-BL}. However the inequality \eqref{eq:sub} itself is not sufficient (even if $k_a >0$) since the desired estimate
$$
a \, \vert \nabla P_t f \vert \leq P_{t,a} ^{\rho_a} (\vert a\nabla f \vert ), \quad t \geq 0,
$$
is stronger and cannot be obtained by using the same approach. Note that it rewrites according to the intertwining relation of Theorem \ref{theo:intert} as
$$
\vert Q_{t,a} ^{M_a} (a \, \nabla f ) \vert \leq P_{t,a} ^{\rho_a} (\vert a\nabla f \vert ).
$$
Here $(P_{t,a} ^{\rho_a})_{t\geq 0}$ is the Feynman-Kac semigroup acting on functions as follows: for every $f\in \C _0 ^\infty (\R^d , \R)$,
$$
P_{t,a} ^{\rho _a} f (x) := \E \left[ f(X_{t,a} ^x) \, e^{- \int_0 ^t \rho _a (X_{s,a} ^x) \, ds }\right] ,
$$
and its generator is given by $L_a ^{\rho_a} = L_a - \rho_a \, I$. Actually, in order to obtain such a stronger estimate above, we need a uniqueness result in the $L^\infty$ Cauchy problem, i.e. when the solution is bounded, and this is the reason why we have to assume the non-explosion of the underlying process $(X_{t,a} ^x)_{t\geq 0}$, an assumption encoded in the statement of Theorem \ref{thm:asym-BL} by a uniform lower bound on the matrix $\nabla \nabla V_a$ or a Lyapunov condition (recall that by \cite{khas} the non-explosion is classically equivalent to the uniqueness in the $L^\infty$ Cauchy problem). To follow this strategy, we use the stochastic representation of the semigroup $(Q_{t,a} ^{M_a})_{t\geq 0}$. If the process $(X_{t,a} ^x)_{t\geq 0}$ does not explode, i.e. it has an infinite lifetime, we consider the matrix-valued process $(Y_{t,a,x} )_{t\geq 0} \subset \M _{d\times d } (\R )$ solution to the equation
$$
\partial_t Y_{t,a, x} +  Y_{t,a, x} \, M_a (X_{t,a} ^x )  = 0, \quad Y_{0,a, x} = I .
$$
We first have the following lemma, allowing us to control the Euclidean operator norm $\vert \cdot \vert _{\rm{op}}$ on $\M _{d\times d } (\R)$ of the matrix $Y_{t,a, x} $.
\begin{lemme}
\label{lemme:Y}
Assume that the smallest eigenvalue $\rho_a$ of the matrix $M_a$ is bounded from below by some real constant and that the process $(X_{t,a} ^x)_{t\geq 0}$ is non-explosive. Then we have the estimate
$$
\vert Y_{t,a, x} \vert _{\rm{op}} \leq e^{- \int_0 ^t \rho_a (X_{s,a} ^x) \, ds}, \quad t \geq 0, \quad x\in \R ^d .
$$
\end{lemme}
\begin{proof}
We have for every $u\in \R ^d$,
\begin{eqnarray*}
\partial_t \vert Y_{t,a, x} ^T \, u \vert ^2 & = & 2 \, u^T \, Y_{t,a, x} \, \partial _t Y_{t,a, x} ^T \, u \\
& = & - 2 \, u^T \, Y_{t,a, x} \, M_a (X_{t,a} ^x) \, Y_{t,a, x} ^T \, u \\
& \leq & - 2 \, \rho _a (X_{t,a} ^x) \, \vert Y_{t,a, x} ^T \, u \vert ^2 ,
\end{eqnarray*}
so that integrating and taking the supremum over all $u\in \R ^d$ such that $\vert u\vert =1$ entails the inequality
$$
\vert Y_{t,a, x} \vert _{\rm{op}} = \vert Y_{t,a, x} ^T \vert _{\rm{op}} \leq e^{- \int_0 ^t \rho_a (X_{s,a} ^x) \, ds } .
$$
The proof is achieved.
\end{proof}
Next we state the desired uniqueness result in the $L^\infty$ Cauchy problem.
\begin{prop}
\label{prop:feynman}
Assume that the matrix $M_a$ is uniformly bounded from below and that the process $(X_{t,a} ^x)_{t\geq 0}$ is non-explosive. Let $F$ be a smooth vector field which is locally bounded in time and bounded with respect to the space variable. If such a $F$ is solution to the $L^\infty$ Cauchy problem
\[
\left \{
\begin{array}{ccc}
\partial _t F & = & \L _a ^{M_a} F \\
F(\cdot , 0) & = & G
\end{array}
\right.
\]
where the initial condition $G$ is bounded, then $F$ admits the stochastic representation
\begin{equation}
\label{eq:stochastic}
F(x,t) = \E \left[ Y_{t,a,x} \, G(X_{t,a} ^x ) \right] , \quad t \geq 0, \quad x\in \R ^d.
\end{equation}
\end{prop}
\begin{proof}
The proof relies on a martingale method. Let $t >0$ be a finite time horizon and define the $\R^d$-valued process
$$
Z_{s,a,x} := Y_{s,a,x} \, F(X_{s,a} ^x , t-s), \quad s\in [0,t],
$$
where $F$ is such a solution. By Lemma \ref{lemme:Y} and our assumptions the process $(Z_{s,a,x})_{s\in [0,t]}$ is bounded. Since $Z_{0,a,x} = F(x,t)$ and $Z_{t,a,x} = Y_{t,a,x} \,G (X_{t,a} ^x)$, the identity \eqref{eq:stochastic} we want to establish rewrites as
$$
Z_{0,a,x} = \E \left[ Z_{t ,a,x}\right] .
$$
By the vectorial It\^o formula, we get
$$
d Z_{s,a,x} = dM_s + Y_{s,a,x} \, ( - \partial_s + \L _a ) \, F(X_{s,a} ^x , t-s) \, ds - Y_{s,a, x} \, M_a (X_{s,a} ^x ) \, F(X_{s,a} ^x , t-s) \, ds ,
$$
where $(M_s)_{s\in [0,t]}$ is a $\R^d$-valued local martingale. Since we have $\partial_s F = \L _a ^{M_a} F$, the process $(Z_{s,a,x})_{s\in [0,t]}$ itself is a local martingale and actually a true martingale by boundedness. Finally equating the expected values at $s = 0$ and $s = t$ entails the desired result.
\end{proof}

Now we are in position to refine Theorem \ref{thm:compar_feynman}, at the price of the additional assumption of non-explosion of the process $(X_{t,a} ^x)_{t\geq 0}$. Our result stands as follows.
\begin{theo}
\label{theo:feynman_kac}
Assume that the matrix $M_a$ is uniformly bounded from below and that the process $(X_{t,a} ^x)_{t\geq 0}$ is non-explosive. Then for every $f\in \C _0 ^\infty (\R ^d , \R)$, we have the identity
$$
a (x) \, \nabla P_t f (x) = Q_{t,a} ^{M_a} (a\, \nabla f) (x) = \E \left[ Y_{t,a, x} \, a (X_{t,a} ^x) \, \nabla f (X_{t,a} ^x) \right], \quad x\in \R ^d, \quad t \geq 0.
$$
In particular we have a refinement of Theorem \ref{thm:compar_feynman}: for every $f\in \C _0 ^\infty (\R ^d , \R)$, the following sub-intertwining holds:
$$
\vert a \, \nabla P_t f \vert = \vert Q_{t,a} ^{M_a} (a\, \nabla f) \vert \leq P_{t,a} ^{\rho_a} (\vert a \, \nabla f \vert ) .
$$
\end{theo}

\begin{proof}
The proof of the first equality is now straightforward: by the intertwining relation of Theorem \ref{theo:intert}, the quantity $(a \, \nabla P_t f)_{t\geq 0}$ is a solution to the $L^\infty$ Cauchy problem associated to the Schr\"odinger operator $\L_a ^{M_a}$ since it is bounded by Theorem \ref{thm:compar_feynman}. Then the first desired equality is a consequence of the uniqueness property of Proposition \ref{prop:feynman} with $G=a \, \nabla f$. \vspace{0.1cm} \\
To establish the sub-intertwining relation, note that we have for every $F\in \C_0 ^\infty (\R^d , \R)$,
\begin{eqnarray*}
\vert \E [ Y_{t,a, x}  \, F(X_{t,a} ^x)] \vert
& \leq & \E [ \vert Y_{t,a, x}  \, F(X_{t,a} ^x) \vert ] \\
& \leq & \E [ \vert Y_{t,a, x} \vert _{\rm{op}} \, \vert F(X_{t,a} ^x) \vert ] \\
& \leq & \E \left[ e^{-\int_0 ^t \rho _a (X_{s,a} ^x) \, ds } \, \vert F(X_{t,a} ^x) \vert \right] \\
& = & P_{t,a} ^{\rho_a } (\vert F\vert) (x),
\end{eqnarray*}
where to obtain the last inequality we used Lemma \ref{lemme:Y}. The proof is complete
\end{proof}
We mention that a similar sub-intertwining already appeared in the work of Wang \cite{wang} who introduced the notion of modified curvatures to study reflecting diffusion processes on Riemannian manifolds with boundary. \vspace{0.1cm}

Once all the previous ingredients of this section have been introduced, we are ready to give the (brief) proof of Theorem \ref{thm:asym-BL} to which we turn now.
\begin{Pf_th}
At the price of an approximation procedure somewhat similar to that emphasized in the proof of Theorem \ref{theo:BL}, we assume without loss of generality that the smallest eigenvalue $\rho_a$ of the matrix $M_a$ is bounded from below by some positive constant. As mentioned earlier, the covariance version of the variance identity \eqref{eq:var_id} reads as follows: for every $f,g \in \C _0 ^\infty (\R^d , \R)$,
\begin{eqnarray*}
\Cov _\mu (f,g) & = & \int_{\R^d} (a \, \nabla f) ^T \, (- \L _a ^{M_a})^{-1} (a \, \nabla g) \, d\mu_a .
\end{eqnarray*}
Now we have
\begin{eqnarray*}
\vert (-\L _a ^{M_a})^{-1} (a \, \nabla g)\vert & = & \left \vert\int_0^{\infty} \Q_{t,a}^{M_a} (a\, \nabla g) \, dt \right \vert \\
& \leq & \int_0^{\infty} \vert \Q_{t,a}^{M_a} (a\, \nabla g) \vert \, dt \\
& \leq & \int_0^{\infty} P_{t,a} ^{\rho_a} ( \vert a \, \nabla g \vert ) \, dt \\
& = & (- L _a ^{\rho_a})^{-1} ( \vert a \, \nabla g \vert ) ,
\end{eqnarray*}
where to obtain the second inequality we used Theorem \ref{theo:feynman_kac}. Since we have $L_a 1 = 0$, it leads to the equality $(-L _a ^{\rho_a} )^{ -1} \rho_a  =1$ and by positivity preservation we get
$$
\left \Vert (-L_a ^{\rho_a})^{-1} (\vert a \, \nabla g\vert ) \right \Vert _\infty \leq \left \Vert \frac{\vert a \, \nabla g\vert }{\rho_a} \right \Vert _\infty \, \left \Vert (-L_a ^{\rho_a} )^{-1} \rho_a \right \Vert _\infty = \left \Vert \frac{\vert a \, \nabla g \vert }{\rho_a} \right \Vert _\infty  .
$$
The proof of Theorem \ref{thm:asym-BL} is now complete.
\end{Pf_th}

\section{Examples}
\label{sect:ex}

In this final part we illustrate our main results with some classical and less classical examples. Before turning to concrete situations, we would like to emphasize a criterion which is well-adapted to most of the examples presented below. Considering $A= a\, I$ with $a$ the function given by $a := e^{-\ve V}$ for $\ve$ sufficiently small (in fact $\ve \in (0,1/2)$), we obtain with the smooth potential $V_a = V +\log (a^2)$ the new generator
\begin{eqnarray*}
L_a f & = & \Delta f - (\nabla V_a) ^T \, \nabla f \\
& = & \Delta f - (1-2\ve) \, (\nabla V) ^T \, \nabla f .
\end{eqnarray*}
In particular the invariant measure $\mu_a$ has Lebesgue density proportional to $e^{-(1-2\ve) V}$. Therefore we are in position to apply Theorem \ref{theo:BL} once the matrix $A^{-1} \, M_A \, A $, which rewrites as the symmetric matrix $M_a$, is positive definite. Moreover, recall that we assumed all along the paper that the Hessian matrix of $V$ is uniformly bounded from below, a condition ensuring the non-explosion of the underlying process \cite{bakry:non-exp}. Then the Hessian matrix of $V_a$ is also uniformly bounded from below and thus the process $(X_{t,a} ^x)_{t\geq 0}$ is non-explosive, so that Theorem \ref{thm:asym-BL} might also be applied. Computing now the matrix $M_a$ gives us
\begin{equation}
\label{aev}
M_a = \nabla \nabla V +( - \ve \, \Delta V + \ve \, (1 -\ve ) \, \vert \nabla V \vert ^2 ) \, I.
\end{equation}
In the case of one-dimensional diffusions \cite{bjm2}, a simplification occurs so that this choice of function $a$ revealed to be very efficient to establish Poincar\'e type functional inequalities. We will see below that this choice is also relevant in the multi-dimensional case. To be as concise as possible, we will only focus our attention on the results derived from Theorems \ref{theo:BL} and \ref{theo:spectral}. \vspace{0.1cm}

Let us now revisit some examples investigated in \cite{bjm}. We consider the case of some spherically symmetric probability measures, that is, the smooth potential $V$ is radial: we set $ V(x) := U(|x|)$ where $U:[0, \infty) \to \R $. Then the operator $\LL$ rewrites for every $f\in \C ^\infty (\R^d , \R)$ as
$$
\LL f (x) = \Delta f (x) - \frac{U'(\vert x\vert )}{\vert x\vert } \, x ^T \, \nabla f( x), \quad x\in \R ^d .
$$
Simple computations show that we have
$$
\nabla \nabla V (x) = \frac{U'(\vert x\vert )}{\vert x\vert } \, I + \left( U''(\vert x \vert ) - \frac{U'(\vert x \vert )}{\vert x \vert}\right) \, \frac{x \, x^T}{\vert x\vert ^2},
$$
whose eigenvalues are $U''(\vert x \vert )$ and $U'(\vert x\vert )/\vert x\vert $, with respective eigenfunctions $x$ and the vectors $u$ belonging to the orthogonal complement of the vector space spanned by $x$. Therefore if additionally the measure $\mu$ is log-concave, meaning in our radial context that the radial potential $U$ is convex and non-decreasing, then the classical Brascamp-Lieb inequality \eqref{eq:BL} gives for every $f\in \C ^\infty _0 (\R^d , \R)$,
$$
\Var_\mu(f) \leq \int_{\R^d} \frac{\vert \nabla f (x) \vert ^2 }{\min \left \{ U''(\vert x \vert ) , \frac{U'(\vert x \vert )}{\vert x \vert} \right \} } \, d\mu (x).
$$
For instance when we consider the so-called exponential power, or Subbotin, distribution of parameter $\alpha >1$, that is, $U$ is given on $[0,\infty)$ by $U(r) = r^\alpha /\alpha $, then we have
\begin{eqnarray*}
\nabla \nabla V (x) & = & \vert x\vert ^{\alpha -2} \, I + (\alpha -2) \, \vert x\vert ^{\alpha -4} \, x \, x^T ,
\end{eqnarray*}
whose smallest eigenvalue $\rho_1 (x)$ is
$$
\rho_1 (x) = \min \left \{ U''(\vert x \vert ) , \frac{U'(\vert x \vert )}{\vert x \vert} \right \} = \min \{ 1 , \alpha -1\} \,  \vert x\vert^{\alpha -2} .
$$
In other words, this minimum strongly depends on the value of the parameter $\alpha$, according to $\alpha \in (1,2]$ or $\alpha \geq 2$, the critical case being the Gaussian case $\alpha = 2$ for which $U''(r) = U'(r) / r = 1$. \vspace{0.1cm}

Let us see how to apply Theorem \ref{theo:BL} in our general radial context. Letting $A$ be the invertible matrix $A = a \, I$ with $a$ some smooth and positive function, the smallest eigenvalue $\rho_a$ of the symmetric matrix $A ^{-1} \, M_A \, A$ is
$$
\rho_a (x) = \min \left \{ U''(\vert x \vert ) , \frac{U'(\vert x \vert )}{\vert x \vert} \right \} - a (x) \, \LL a^{-1} (x), \quad x \in \R ^d,
$$
and if the condition $\rho_a >0$ holds, then by Theorem \ref{theo:BL} we have the generalized Brascamp-Lieb inequality: for every $f\in \C ^\infty _0 (\R^d , \R)$,
$$
\Var_\mu(f) \leq \int_{\R^d} \frac{\vert \nabla f (x) \vert ^2 }{\min \left \{ U''(\vert x \vert ) , \frac{U'(\vert x \vert )}{\vert x \vert} \right \} - a (x) \, \LL a^{-1} (x)} \, d\mu (x).
$$
Note that the condition $\rho_a >0$ induces a region on which the potential $V$ is strictly convex. Indeed, the function $- a \, \LL a^{-1} $ cannot be positive on $\R^d$ since it is non-positive in average:
\begin{eqnarray*}
- \int_{\R^d} a \, \LL a^{-1}  \, d\mu & = & \int_{\R^d} (\nabla a)^T \, \nabla a^{-1} \, d\mu \\
& = & - \int_{\R ^d} \frac{\vert \nabla a \vert ^2}{a^2} \, d\mu \\
& \leq & 0.
\end{eqnarray*}
Moreover we also observe that if $\inf \, \rho_a >0$ then we have by Theorem \ref{theo:spectral} the spectral gap estimate
$$
\lambda_1 (-\LL , \mu) \geq \inf \, \rho_a ,
$$
giving an alternative criterion to the estimates obtained in \cite{bjm} for spherically symmetric log-concave probability measures. \vspace{0.1cm}

Now the point is to find a nice function $a$ such that $\rho_a >0$ and, as announced, we choose $a = e^{-\varepsilon V }$ where $\varepsilon \in (0,1/2)$, so that by \eqref{aev} the condition $\rho_a >0$ becomes: for every $x\in \R^d$,
\begin{equation}
\label{eq:spher_sym2}
\min \left \{ U''(\vert x \vert ) , \frac{U'(\vert x \vert )}{\vert x \vert} \right \} - \varepsilon (d-1) \, \frac{U'(\vert x \vert )}{\vert x \vert}  - \varepsilon \, U''(\vert x \vert ) + \varepsilon \, (1-\varepsilon) \, U'(\vert x \vert ) ^2 >0.
\end{equation}
Let us see how this criterion might be applied for some particular potentials $U$. Coming back to the previous example of the exponential power distribution of parameter $\alpha >1$, we denote $\gamma := (\alpha +d -2)/ \min \{ 1 , \alpha -1\} $ (which is $\geq 2$ since $d\geq 2$) and we have
\begin{eqnarray*}
\rho_a(x) & =  & \vert x \vert ^{\alpha -2} \, \left( \min \{ 1 , \alpha -1\}  - \varepsilon \, (\alpha +d -2) \right) + \varepsilon \, (1-\varepsilon) \, \vert x \vert ^{2(\alpha -1)} \\
& \geq & \min \left \{ 1 - \varepsilon \, \gamma , \varepsilon \, (1-\varepsilon) \right \} \, \left( \min \{ 1 , \alpha -1\} \, \vert x \vert ^{\alpha -2} + \vert x \vert ^{2(\alpha -1)} \right) .
\end{eqnarray*}
Hence optimizing in $\varepsilon \in (0,1/\gamma)$ yields the inequality
$$
\rho_a(x) \geq \frac{8\sqrt{\gamma -1}}{(\sqrt{\gamma-1} + \sqrt{\gamma+3})^3} \, \left( \min \{ 1 , \alpha -1\} \, \vert x \vert ^{\alpha -2} + \vert x \vert ^{2(\alpha -1)} \right) ,
$$
the prefactor behaving in large dimension as
$$
\frac{8\sqrt{\gamma -1}}{(\sqrt{\gamma-1} + \sqrt{\gamma+3})^3} \underset{d \to \infty}{\approx} \frac{\min \{ 1 , \alpha -1\}}{d}.
$$
Then the following generalized Brascamp-Lieb inequality holds: for every $f\in \C ^\infty _0 (\R^d , \R)$,
$$
\frac{8\sqrt{\gamma -1}}{(\sqrt{\gamma-1} + \sqrt{\gamma+3})^3} \, \Var_\mu(f) \leq \int_{\R^d} \frac{\vert \nabla f (x) \vert ^2 }{\min \{ 1 , \alpha -1\} \, \vert x \vert ^{\alpha -2} + \vert x \vert ^{2(\alpha -1)}} \, d\mu (x).
$$
In the Gaussian case $\alpha = 2$ we have $\gamma = d$ and we obtain the nice inequality
$$
\frac{8\sqrt{d-1}}{(\sqrt{d-1} + \sqrt{d+3})^3} \, \Var_\mu(f) \leq \int_{\R^d} \frac{\vert \nabla f (x) \vert ^2 }{1 + \vert x \vert ^2} \, d\mu (x),
$$
which is asymptotically sharp as the dimension $d$ goes to infinity, as observed by computing both sides of the inequality with the linear function $f(x) = \sum_{k=1} ^d x_k$. It even slightly improves the constant of the same estimate obtained in \cite{bjm} through another approach. \vspace{0.1cm}

Let us continue the study of generalized Brascamp-Lieb inequalities for spherically symmetric probability measures by considering the case of a heavy-tailed distribution which is, in essence, rather different from the log-concave measures previously investigated. More precisely we focus on the so-called generalized Cauchy distribution of parameter $\beta$, that is, the probability measure $\mu$ has Lebesgue density on $\R^d$ proportional to $(1+\vert x\vert ^2) ^{-\beta}$ where $\beta >d/2$.
Then the associated potential $V$ is radial and is given by $V(x)=U(|x|)$ with
$$
U(r) = \beta \, \log (1+ r^2) , \quad r \geq 0.
$$
The Hessian matrix of $V$ is
$$
\nabla \nabla V (x) = \frac{2 \beta}{1+\vert x\vert ^2} \, I - \frac{4 \beta}{(1+\vert x\vert ^2)^2} \, x \, x^T ,
$$
whose smallest eigenvalue $\rho_1 (x)$ is
$$
\rho_1 (x) = \frac{2\beta \, (1-\vert x\vert ^2)}{(1+\vert x\vert ^2)^2} ,
$$
which is not positive but is bounded from below by $-\beta /4$. Using \eqref{eq:spher_sym2} we have
\begin{eqnarray*}
\rho_a (x) & = & \frac{2 \beta}{(1+\vert x\vert ^2)^2} \, \left( 1 - \varepsilon d + (-1 - \varepsilon d + 2 \varepsilon + 2\beta \varepsilon \, (1-\varepsilon)) \, \vert x\vert ^2 \right) \\
& \geq & \frac{2\beta}{1+\vert x\vert ^2} \, \min \left \{ 1 - \varepsilon d, -1 - \varepsilon d +2 \varepsilon + 2\beta \varepsilon \, (1-\varepsilon)\right \},
\end{eqnarray*}
and optimizing in $\varepsilon \in (0,1/d)$ entails the estimate
$$
\rho_a (x) \geq \frac{2 \, (\beta - d)}{1+\vert x\vert ^2} ,
$$
leading to the generalized Brascamp-Lieb inequality
$$
2 \, (\beta - d) \, \Var_\mu(f) \leq \int_{\R^d} (1 + \vert x \vert ^2) \, \vert \nabla f (x) \vert ^2 \, d\mu (x).
$$
As shown by Nguyen in \cite{nguyen}, the optimal constant in the above inequality is $2 \, (\beta - 1)$ for $\beta > d+1$, the case $\beta \in (d/2 , d+1)$ being covered by the results emphasized in \cite{bjm}. Hence the latter estimate is sharp in dimension 1 for $\beta > 3/2$ but there is still room for improvement in larger dimension when using the intertwining method, maybe by choosing conveniently another function $a$. However we draw the reader's attention to the fact that our estimate is obtained with exactly the same argument as in the log-concave case above, showing the relevance of the intertwining method in a wide range of situations. \vspace{0.1cm}

Let us finish this work by observing how the intertwining method allows us to obtain a spectral gap estimate beyond the case of spherically symmetric probability measures. In particular we concentrate on a non-classical example: a non-product measure on $\R^2$ whose associated potential $V$ exhibits a non compact region on which it is concave. More precisely we consider the following potential, symmetric in both coordinates $x$ and $y$,
$$
V(x) : = \frac{x^4}{4}+ \frac{y^4}{4} - \beta x y , \quad x,y \in \R ,
$$
where $\beta$ is some positive parameter controlling the size of the concave region. Although our approach might be generalized to larger dimension, we reduce the study to dimension 2 since our objective is to give a flavour of how the intertwining method can be used to obtain a spectral gap estimate. In particular the constants we obtain below have no reason to be sharp. To further complement the spectral analysis of these type of models appearing in statistical mechanics, see for instance, among many other interesting articles, the nice papers \cite{helffer, helffer3, ledoux, gr, chen}. \vspace{0.1cm}

In order to use Theorem \ref{theo:spectral}, we have to find a suitable invertible matrix $A$ such that the matrix $A^{-1} \, M_A \, A$ is symmetric and its smallest eigenvalue $\rho_a$ is bounded from below by some positive constant. In contrast to the previous examples, the choice of the matrix $A$ as a multiple of the identity is not convenient because of the degeneracy of the Hessian matrix of $V$:
\[
          \nabla \nabla V(x,y) = \left( \begin{array}{cc}
                                  3x^2 &-\beta  \\
                                  -\beta &  3 y^2 \\
                                 \end{array}\right) .
\]
The idea is to choose a diagonal matrix $A$ with different weight on the diagonal, allowing us to overcome this degeneracy. Set
\[
          A = \left( \begin{array}{cc}
                                 a_1 & 0  \\
                                  0 & a_2 \\
                                 \end{array}\right) ,
\]
where $a_1 , a_2$ are two positive smooth functions on $\R ^2$. Then the matrix $A^{-1} \, M_A \, A$ is of the following form
\[
A^{-1} \, M_A \, A = \nabla \nabla V - \L A^{-1} \, A = \left(
\begin{array}{cc}
X & -\beta  \\
-\beta & Y \\
\end{array}
\right) ,
\]
with
$$
X (x,y):= 3x^2 - a_1 \, \LL a_1 ^{-1} \quad \mbox{ and } \quad Y(x,y):= 3y^2 - a_2 \, \LL a_2 ^{-1} .
$$
The eigenvalues are thus given by
$$
\frac{X+Y}{2} \pm \frac{1}{2} \sqrt{(X-Y)^2 + 4 \beta ^2} ,
$$
which are both bounded from below by $\min \{ X,Y \} - \beta$. Hence choosing $a_2 (x,y) = a_1 (y,x)$, we have $\inf\, \rho_a >0$ as soon as $\inf X - \beta >0$ and the rest of the analysis is devoted to find a convenient function $a_1 = a$ (we drop the subscript 1 in the sequel) ensuring this condition. Letting $a = e^{-W}$ where $W$ is some smooth function defined on $\R^2$, we have
\begin{eqnarray*}
3x^2 - a (x,y) \, \LL a ^{-1} (x,y) & = & 3x^2 + \Delta W (x,y)- \vert \nabla W (x,y)\vert ^2 - \nabla V(x,y) \,\nabla W(x,y) \\
& = &  3x^2 + \Delta Z(x,y) - \vert \nabla Z(x,y) \vert ^2 - \frac{\Delta V(x,y)}{2} + \frac{\vert \nabla V(x,y) \vert ^2}{4} ,
\end{eqnarray*}
where $Z := W + V/2$. The previous equation indicates that a potential candidate is
$$
Z(x,y) := \frac{b y^4}{4} + \frac{c x^2}{2},
$$
where $b,c$ are some real constants to be chosen at the end. A short computation gives
\begin{eqnarray*}
\lefteqn{3x^2 + \Delta Z(x,y) - \vert \nabla Z(x,y) \vert ^2 - \frac{\Delta V(x,y)}{2} + \frac{\vert \nabla V(x,y) \vert ^2}{4} } \\
& & = 3x^2 + c+ 3 by^2 -( c^2 x^2 + b^2 y^6) - \frac{3(x^2+y^2)}{2} + \frac{(x^3 -\beta y)^2}{4} + \frac{(y^3- \beta x)^2}{4} \\
& & \geq c + \left(\frac{3}{2}-c^2\right) x^2 + 3 \, \left( b -\frac{1}{2}\right) y^2 -b^2 y^6 +\frac{(y^3- \beta x)^2}{4} \\
&  & \geq c + \left(\frac{3}{2}-c^2\right) x^2 + 3 \, \left( b -\frac{1}{2}\right) y^2 -b^2 y^6 +\frac{(1-\lambda )\, y^6}{4}- \left( \frac{1}{\lambda} - 1 \right) \, \beta^2 x^2 ,
\end{eqnarray*}
where in the last line we used the trivial inequality $(u-v)^2 \geq (1-\lambda ) \, u^2 - (1/\lambda -1) \, v^2 $ available for every $u,v\in \R$ and $\lambda \in (0,1)$. Now since for every $p,q>0$ the minimum on $\R_+$ of the function $z \mapsto -p \, z^2 + q \, z^3$ is
\[
 -\frac{2 p \, \sqrt{p}}{3 \, \sqrt{3q}},
\]
we obtain from the above inequality,
\begin{eqnarray*}
\lefteqn{3x^2 + \Delta Z(x,y) - \vert \nabla Z(x,y) \vert ^2 - \frac{\Delta V(x,y)}{2} + \frac{\vert \nabla V(x,y) \vert ^2}{4} } \\
& & \geq c + \left(\frac{3}{2}-c^2 - \left( \frac{1}{\lambda} - 1 \right) \, \beta^2 \right) \, x^2 - 3 \, \left( \frac{1}{2} - b\right) \, y^2 + \left( \frac{1-\lambda}{4} - b^2 \right) \, y^6 \\
& & \geq c + \left(\frac{3}{2}-c^2 - \left( \frac{1}{\lambda} - 1 \right) \, \beta^2 \right) \, x^2 - \frac{2 \, \left( \frac{1}{2} - b \right) ^{3/2}}{\sqrt{\frac{1-\lambda}{4} - b^2}} ,
\end{eqnarray*}
provided the constant $b$ satisfies $\vert b\vert < \sqrt{(1-\lambda)} /2$. For instance taking $\lambda= 1/2$ and $b=1/4$ yields
\begin{eqnarray*}
3x^2 + \Delta Z(x,y) - \vert \nabla Z(x,y) \vert ^2 - \frac{\Delta V(x,y)}{2} + \frac{\vert \nabla V(x,y) \vert ^2}{4}  & \geq & c + \left(\frac{3}{2}-c^2 - \beta^2 \right) \, x^2 - 1,
\end{eqnarray*}
and choosing $c := \sqrt{3/2 - \beta^2} $ entails for sufficiently small $\beta$ the estimate
\begin{eqnarray*}
\inf \rho_a & \geq & \sqrt{3/2 - \beta^2} - 1 - \beta \, > \, 0,
\end{eqnarray*}
leading by Theorem \ref{theo:spectral} to the conclusion
$$
\lambda_1 (-L , \mu) \geq \sqrt{3/2 - \beta^2} - 1 - \beta .
$$


\vspace{0.5cm}

\begin{thebibliography}{10}

\bibitem{bakry:non-exp} D. Bakry.
Un crit\`ere de non-explosion pour certaines diffusions sur une vari\'et\'e riemannienne compl\`ete. C. R. Acad. Sci. Paris S\'er. I Math. 303:23-26, 1986.

\bibitem{bakry_emery} D. Bakry and M. \'Emery. Diffusions hypercontractives. \textit{S\'eminaire de Probabilit\'es}, XIX, 177-206, Lecture Notes in Math., 1123, Springer, 1985.

\bibitem{BGL} D. Bakry, I. Gentil, and M. Ledoux. \textit{Analysis and geometry of Markov diffusion operators}. Grundlehren der mathematischen Wissenschaften, 348, Springer, Heidelberg, 2013.


\bibitem{bobkov} S.G. Bobkov. Isoperimetric and analytic inequalities for log-concave probability measures. \textit{Ann. Probab.}, 27:1903-1921, 1999.

\bibitem{bht} S.G. Bobkov, F. G\"otze and C. Houdr\'e. On {G}aussian and {B}ernoulli covariance representations. \textit{Bernoulli}, 7:439-451, 2001.

\bibitem{bob_ledoux3} S.G. Bobkov and M. Ledoux. From {Brunn}-{M}inkowski to {B}rascamp-{L}ieb and to logarithmic {S}obolev inequalities. \textit{Geom. Funct. Anal.}, 10:1028-1052, 2000.

\bibitem{bob_ledoux2} S.G. Bobkov and M. Ledoux. Weighted {P}oincar\'e-type inequalities for {C}auchy and other convex measures. \textit{Ann. Probab.}, 37:403-427, 2009.

\bibitem{bgg} F. Bolley, I. Gentil and A. Guillin. Dimensional improvements of the logarithmic {S}obolev, {T}alagrand and {B}rascamp-{L}ieb inequalities. Preprint 2015.

\bibitem{bj} M. Bonnefont and A. Joulin. Intertwining relations for one-dimensional diffusions and application to functional inequalities. \textit{Pot. Anal.}, 41:1005-1031, 2014.

\bibitem{bjm} M. Bonnefont, A. Joulin and Y. Ma. Spectral gap for spherically symmetric log-concave probability measures, and beyond. Preprint 2015.

\bibitem{bjm2} M. Bonnefont, A. Joulin and Y. Ma. A note on spectral gap and weighted {P}oincar\'e inequalities for some one-dimensional diffusions. To appear in \textit{ESAIM Probab. Stat.}, 2016.

\bibitem{brascamp_lieb} H.J. Brascamp and E.H. Lieb. On extensions of the {B}runn-{M}inkovski and {P}r\'ekopa-{L}eindler theorems, including inequalities for log-concave functions, and with an application to the diffusion equation. \textit{J. Funct. Anal.}, 22:366-389, 1976.

\bibitem{carlen_cordero} E. Carlen, D. Cordero-Erausquin and E. Lieb. Asymmetric covariance estimates of {B}rascamp-{L}ieb type and related inequalities for log-concave measures. \textit{Ann. Inst. H. Poincar\'e Probab. Statist.}, 49:1-12, 2013.

\bibitem{cj} D. Chafa\"i and A. Joulin. Intertwining and commutation relations for birth-death processes. \textit{Bernoulli}, 19:1855-1879, 2013.

\bibitem{chen} M.F. Chen. Spectral gap and logarithmic {S}obolev constant for continuous spin systems. \textit{Acta Math. Sin. English Series}, 24:705-736, 2008.

\bibitem{chen_wang} M.F. Chen and F.Y. Wang. Estimation of spectral gap for elliptic operators. \emph{Trans. Amer. Math. Soc.}, 349:1239-1267, 1997.

\bibitem{cordero} D. Cordero-Erausquin. Transport inequalities for log-concave measures, quantitative forms and applications. Preprint 2015.

\bibitem{gr} I. Gentil and C. Roberto. Spectral gaps for spin systems: some non-convex phase examples. \textit{J. Funct. Anal.}, 180:66-84, 2001.

\bibitem{harge} G. Harg\'e. Reinforcement of an inequality due to {B}rascamp and {L}ieb. \textit{J. Funct. Anal.}, 254:267-300, 2008.

\bibitem{helffer} B. Helffer. Remarks on decay of correlations and {W}itten {L}aplacians, {B}rascamp-{L}ieb inequalities and semiclassical limit. \textit{J. Funct. Anal.}, 155:571-586, 1998.

\bibitem{helffer3} B. Helffer. Remarks on decay of correlations and {W}itten {L}aplacians III - {A}pplication to logarithmic {S}obolev inequalites. \textit{Ann. Inst. H. Poincar\'e Probab. Statist.}, 35:483-508, 1999.

\bibitem{helffer_book} B. Helffer. \emph{Semiclassical analysis, {W}itten {L}aplacians, and statistical mechanics}. Series in Partial Differential Equations and Applications, World Scientific Publishing, 2002.

\bibitem{hormander} L. H\"ormander. $L^2$ estimate and existence theorems for the $\bar{\partial}$ operator. \textit{Acta Math.}, 113:89-152, 1965.

\bibitem{houdre} C. Houdr\'e. Remarks on deviation inequalities for functions of infinitely divisible random vectors. \textit{Ann. Probab.}, 30:1223-1237, 2002.

\bibitem{kls} E. Kannan, L. Lov\'asz, and M. Simonovits. Isoperimetric problems for convex bodies and a localization lemma. \textit{Discrete Comput. Geom.}, 13: 541-559, 1995.

\bibitem{khas} R.Z. Khas'minskii. Ergodic properties of recurrent diffusion processes and stabilization of solution to the Cauchy problem for parabolic equations. \textit{Theor. Prob. Appl.}, 5:179-195, 1960.

\bibitem{milman_koles} A.V. Kolesnikov and E. Milman. Poincar\'e and {B}runn-{M}inkowski inequalities on weighted Riemannian manifolds with boundary. Preprint, 2014.

\bibitem{ledoux} M. Ledoux. Logarithmic {S}obolev inequalities for unbounded spin systems revisited. \textit{S\'eminaire de Probabilit\'es}, XXXV, 167-194, Lecture Notes in Math., 1755, Springer, 2001.

\bibitem{Li} P. Li. Uniqueness of $L^1$ solutions for the Laplace equation and the heat equation on Riemannian manifolds. \textit{J. Differential Geometry}, 20:447-457, 1984.

\bibitem{menz} G. Menz. A {B}rascamp-{L}ieb type covariance estimate. \textit{Electron. J. Probab.}, 19:1-15, 2014.

\bibitem{menz_otto} G. Menz and F. Otto. Uniform logarithmic {S}obolev inequalities for conservative spin systems with super-quadratic single-site potential. \textit{Ann. Probab.}, 41:2182-2224, 2013.

\bibitem{meyn_tweedie} S.P. Meyn and R.L. Tweedie. Stability of {M}arkovian processes III: {F}oster-{L}yapunov criteria for continuous-time processes. \textit{Adv. App. Probab.}, 25:518-548, 1993.

\bibitem{milman} E. Milman. On the role of convexity in isoperimetry, spectral gap and concentration. \textit{Invent. Math.}, 177:1-43, 2009.

\bibitem{nguyen} V.H. Nguyen. Dimensional variance inequalities of {B}rascamp-{L}ieb type and a local approach to dimensional {P}r\'ekopa's theorem. \textit{J. Funct. Anal.}, 266:931-955, 2014.

\bibitem{Strichartz} R. Strichartz. Analysis of the Laplacian on the complete Riemannian manifold. \textit{J. Funct. Anal.}, 52:48-79, 1983.

\bibitem{veysseire} L. Veysseire. A harmonic mean bound for the spectral gap of the Laplacian on Riemannian manifolds. \emph{C. R. Math. Acad. Sci. Paris}, 348:1319-1322, 2010.

\bibitem{wang} F.Y. Wang. Modified curvatures on manifolds with boundary and applications. \textit{Pot. Anal.}, 41:699-714, 2014.

\end{thebibliography}
\end{document}